\documentclass[11pt]{article}
\usepackage{amsmath,amsthm,amssymb,hyperref, color,fullpage}
\usepackage{blkarray}
\usepackage{verbatim}

\newcommand{\remove}[1]{}

\makeatletter
\newtheorem*{rep@theorem}{\rep@title}
\newcommand{\newreptheorem}[2]{%
\newenvironment{rep#1}[1]{%
 \def\rep@title{#2 \ref{##1}}%
 \begin{rep@theorem}}%
 {\end{rep@theorem}}}
\makeatother

\hypersetup{
	colorlinks=true,
	linkcolor=blue,%
	citecolor=blue,
	linktoc=page
}
\newcommand\numberthis{\addtocounter{equation}{1}\tag{\theequation}}

\newtheorem{thm}{Theorem}[section]
\newreptheorem{thm}{Theorem}
\newtheorem{claim}[thm]{Claim}
\newtheorem{lem}[thm]{Lemma}
\newtheorem{define}[thm]{Definition}
\newtheorem{cor}[thm]{Corollary}

\newtheorem{example}[thm]{Example}

\newtheorem{prop}[thm]{Proposition}

\def\B{{\mathcal B}}

\def\PATH{{\textsf{PATH}}}

\def\F{{\mathbb{F}}}
\def\E{{\mathbb{E}}}
\def\kF{{\overline{\mathbb{F}}}}

\def\kR{{\overline{R}}}
\def\ce{{\hat{e}}}

\def\Z{{\mathbb{Z}}}
\def\N{{\mathbb{N}}}
\def\R{{\mathbb{R}}}

\def\V{{\mathbf{V}}}
\def\I{{\mathbf{I}}}

\def\E{{\mathbb E}}

\def\_{\,\,\,\,\,}

\usepackage{leftidx}

\usepackage{tikz-cd}

\begin{document}

\title{Furstenberg sets in finite fields: Explaining and improving the Ellenberg-Erman proof}
\author{Manik Dhar\thanks{Department of Computer Science, Princeton University. Email: \texttt{manikd@princeton.edu}. Research supported by NSF grant DMS-1953807.} \and  Zeev Dvir\thanks{Department of Computer Science and Department of Mathematics,
Princeton University.
Email: \texttt{zeev.dvir@gmail.com}. Research supported by NSF CAREER award DMS-1451191 and NSF grant CCF-1523816 and DMS-1953807.} \and Ben Lund\thanks{Discrete Mathematics Group, Institute for Basic Science (IBS), Daejeon, Republic of Korea. Email: \texttt{lund.ben@gmail.com}. Research supported by NSF postdoctoral fellowship DMS-1802787 and the Institute for Basic Science IBS-R029-C1.}}
\date{}

\maketitle

\begin{abstract}
A $(k,m)$-Furstenberg set is a subset $S \subset \F_q^n$ with the property that each $k$-dimensional subspace of $\F_q^n$ can be translated so that it intersects $S$ in at least $m$ points.
Ellenberg and Erman \cite{ellenberg2016furstenberg} proved that $(k,m)$-Furstenberg sets must have size at least $C_{n,k}m^{n/k}$, where $C_{n,k}$ is a constant  depending only $n$ and $k$.
In this paper, we adopt the same proof strategy as Ellenberg and Erman, but use more elementary techniques than their scheme-theoretic method.
By modifying certain parts of the argument, we obtain an improved bound on $C_{n,k}$, and our improved bound is nearly optimal for an algebraic generalization the main combinatorial result.
We also extend our analysis to give lower bounds for sets that have large intersection with shifts of a specific family of higher-degree co-dimension $n-k$ varieties, instead of just co-dimension $n-k$ subspaces.
\end{abstract}

\section{Introduction}

Let $\F_q$ be a finite field, where $q$ is a prime power.
Given a finite set of points $S \subset \F_q^n$, a $k$-dimensional affine subspace $V$ in $\F^n_q$ is {\em $(S,m)$-rich} if $|S\cap V|\ge m$.

\begin{define}[Furstenberg Sets]
A set $S\subseteq \F^n_q$ is {\em $(k,m)$-Furstenberg}, if for each $k$-dimensional linear subspace $V$, some translate of $V$ is $(S,m)$-rich.
\end{define}

The study of finite field Furstenberg sets can be traced back to a question posed by Wolff \cite{wolf1999}. He asked whether a set in $\mathbb{F}_q^n$ that contains a line in each direction must have at least $C_n q^n$ points, where $C_n$ is a positive constant that depends only on $n$.
This question, known as the finite field Kakeya problem, is a combinatorial version of the notoriously difficult Euclidean Kakeya problem that is linked to  fundamental questions in harmonic analysis.
Dvir \cite{dvir2009size} answered Wolff's question affirmatively using the polynomial method, and his proof immediately implies the following bound on Furstenberg sets for lines.

\begin{thm}\label{KakeyaTheorem}
A $(1,m)$-Furstenberg set $S\subseteq \F^n_q$ satisfies the bound,
\begin{align*}
 |S| &\ge C_n m^{n}
\end{align*}
where $C_n$ depends only on $n$. 
\end{thm}

Dvir proved Theorem \ref{KakeyaTheorem} with $C_n = 1/n!$.
Subsequent work \cite{saraf2008, dvir2013extensions} improved the value of $C_n$, and the best possible value of $C_n = 2^{-n+1}$ was recently proved by Bukh and Chao \cite{BC21}.
For a more detailed discussion of this result and its applications, see the survey article \cite{dvir-survey}.

Ellenberg, Oberlin, and Tao \cite{ellenberg_oberlin_tao_2010} used a modification of Dvir's argument to show that, if $S$ is a $(k,q^k)$-Furstenberg set, then $|S| \geq (1-o_q(1))q^n$, where $o_q(1)$ denotes a function $f(q)$ such that $\lim_{q \rightarrow \infty} f(q) = 0$.
Kopparty, Lev, Saraf, and Sudan \cite{kopparty2011kakeya} later improved the $o_q(1)$ term.
However, unlike in the $k=1$ case, the proof of this result does not generalize for $m$ much smaller than $q^k$.

Ellenberg and Erman gave the first non-trivial, fully general lower bound on the size of finite field Furstenberg sets.
In particular, they showed that, if $S$ is a $(k,m)$-Furstenberg set in $\F_q^n$, then $|S| \geq C_{n,k}m^{n/k}$, where $C_{n,k}$ is a constant that depends only on $n$ and $k$.
Instead of bounding the size of Furstenberg sets directly, they introduced and worked on more general Furstenberg schemes; see definition \ref{def:furstenbergAlgebra} for an equivalent definition of {\em Furstenberg algebras} that avoids the language of schemes.

While Ellenberg and Erman didn't explicitly specify the value of $C_{n,k}$, a close inspection of their proof gives $C_{n,k} = (1/n)^{O(n \ln(n/k)}$.
Our first result is a quantitative improvement to Ellenberg and Erman's result.
Our sharpening of their bound for Furstenberg algebras leads to the following improvement to their bound for Furstenberg sets.

\begin{thm}[Furstenberg Set Bound]\label{SetTheorem}
A $(k,m)$-Furstenberg set $S\subseteq \F^n_q$ satisfies the bound,
\begin{align*}
 |S| &\ge C_{n,k} m^{n/k}
\end{align*}
where $C_{n,k}=\Omega((1/16)^{n\ln(n/k)})$. 
\end{thm}

The proof of Theorem \ref{SetTheorem} generally the same as the proof of Ellenberg and Erman, but we avoid the language of schemes to give a more elementary presentation.
The modifications of the proof that lead to the quantitative improvement are discussed in Section \ref{sec:comparison}.

As with the work of Ellenberg and Erman, Theorem \ref{SetTheorem} is a corollary to a more general result on Furstenberg algebras; see Theorem \ref{AlgebraBound}.
The value of $C_{n,k}$ given in Theorem \ref{SetTheorem} is nearly optimal for Furstenberg algebras.
Indeed, Example \ref{constantBound} shows that $C_{n,k}$ in Theorem \ref{AlgebraBound} cannot be larger than $O(e^{-n \log(n/k)})$.

For Furstenberg sets, a simple combinatorial argument (unrelated to \cite{ellenberg2016furstenberg}) gives the following bound, which is superior to Theorem \ref{SetTheorem}  for $m = o(q^{k-k/(k+1)})$.

\begin{thm}[Easy Furstenberg Bound]\label{th:combinatorial}
If $S \subset \F_q^n$ is a (k,m)-Furstenberg set in $\F_q^n$ and $\ell$ is an integer with $1 \leq \ell < \log_q(m)+1$, then
\[ |S|^{\ell +1} \geq q^{\ell(n-k)}m(m-1)(m-q)\ldots(m-q^{\ell-1}) .\]
\end{thm}

Recent work by the current authors \cite{DDL-2} uses completely different techniques to prove Theorem \ref{SetTheorem} with $C_{n,k} = 2^{-n}$.
As with Theorem \ref{th:combinatorial}, the bound proved in \cite{DDL-2} is not true for Furstenberg algebras, and the techniques used there cannot lead to any improvement in Theorem \ref{AlgebraBound}.
It is also unlikely that the techniques of \cite{DDL-2} can be adapted to prove higher-degree analogs of Theorem \ref{SetTheorem}, such as Theorem \ref{thm:hyperBound}.

Our final result is a generalization of Theorem \ref{SetTheorem} for higher degree surfaces. First, for the sake of simplicity we restrict ourselves to the $k=n-1$ case. 
In an $(n-1,m)$-Furstenberg set $S \subset \F_q^n$, for any hyperplane equation $h(x) = a_1x_1 + \ldots + a_nx_n$ there is a constant $c \in \F_q$ such that the equation $h(x) = c$ has at least $m$ solutions in $S$. 
A higher degree analog of this property is that, for any homogeneous degree $d$ equation $h(x)$, there is an equation $f(x)$ of degree at most $d-1$ such that the equation $h(x) = f(x)$ has at least $m$ solutions in $S$. We show that such sets must be large, even if we only require the property to hold for $h(x)$ that are $d$'th power of a hyperplane. The proof turns out to be quite simple given all the machinery already developed to tackle the linear case.  This is defined more formally below.

A {\em hypersurface} in $\F_q^n$ is defined as a zero set of some polynomial in $\F_q[x_1,\hdots,x_n]$. Given a subset $S\subseteq \F_q^n$, a hypersurface is called {\em $(S,m)$-rich} if it contains $m$ many points from $S$. We can generalize this further for higher co-dimension varieties. A set of polynomials $p_1,\hdots,p_k$ is said to be $(S,m)$-rich for a set $S\subseteq \F_q^n$ if $|\{x\in S| p_1(x)=\hdots=p_k(x)=0\}|\ge m$. 

\begin{define}[Hyper-Furstenberg Sets]
A set $S\subseteq \F^n_q$ is  {\em $(k,m,d)$-Hyper-Furstenberg} if, for any $k$-dimensional subspace $U\subseteq \F_q^n$, there exist linearly independent hyperplanes $h_1,\hdots,h_{n-k}$ that contain $U$ and polynomials $g_i\in \F_q[x_1,\hdots,x_n]$ for $i=1,\hdots,n-k$ of degree at most $d_i-1$ with $\prod_{i=1}^{n-k} d_i\le d$ such that the set of polynomials $h_1^{d_1}+g_1,\hdots,h_k^{d_k}+g_k$ is $(S,m)$-rich.
\end{define}

We see that $(n-k,m,d)$-Hyper-Furstenberg sets have large intersections with shifts of a special family of degree $d$ and co-dimension $k$ varieties. For $k=1$ we see that we are looking at shifts of degree $d$ hyper-surfaces. 

\begin{thm}[Hyper-Furstenberg Bound]\label{thm:hyperBound}
A $(k,m,d)$-Hyper-Furstenberg Set $S\subseteq \F^n_q$ satisfies the bound,
$$|S|\ge C_{n,k}\left(\frac{m}{d}\right)^{n/k},$$
where $C_{n,k}=\Omega(1/16^{n\ln(n/k)})$.
\end{thm}

\subsection{Key modifications to the proof of  Ellenberg-Erman}\label{sec:comparison}

For the sake of readers familiar with \cite{ellenberg2016furstenberg}, we outline the main modifications we make to their proof (readers unfamiliar with \cite{ellenberg2016furstenberg} are encouraged to skip this discussion). We only mention those modifications that allow us to improve the quantitative bounds (i.e., not just using different language). 

The improvement to the constant $C_{n,k}$ in Theorem~\ref{SetTheorem} follows from a more significant looking improvement to an intermediate result, Theorem~\ref{thm:furstPlaneBound} in our paper (which appears in \cite{ellenberg2016furstenberg} as the case $k=n-1$ of Theorem 1.5). 
This intermediate result deals with the more general object  of Furstenberg Algebra (or Furstenberg scheme in \cite{ellenberg2016furstenberg}). Theorem~\ref{SetTheorem} follows from this intermediate result by an inductive argument and the improvement to the constant  carries over in the reduction. Ellenberg-Erman prove Theorem~\ref{thm:furstPlaneBound} with constant $1/n$, whereas we are able to improve it to an absolute constant $1/16$. The reduction from Theorem~\ref{SetTheorem} translates this improvement to the final $C_{n,k}=\Omega((1/16)^{n\ln(n/k)})$.

The first main modification is in the part of the proof dealing with Borel-fixed subsets of the integer lattice (following the degeneration to a generic initial ideal in the case $X_{m,k}^S = {\mathrm Gr}(k, n)$). In Lemma 5.3 of \cite{ellenberg2016furstenberg} these are dealt with using a rather short proof by induction. In Lemma~\ref{latticeBound} we give a  different, more involved, treatment of such sets which results in an improved quantitative bound.

The other main modification is the place in \cite{ellenberg2016furstenberg} which uses a deep result of Hochster and Huneke \cite{Hochster2002} to bound the number of $m$-rich flats in the case  $X_{m,k}^S \neq {\mathrm Gr}(k, n)$. In Section~\ref{sec:MultCase} we replace this with the more elementary Schwartz-Zippel lemma (with multiplicities) to arrive at a better quantitative bound.

\subsection{Organization}

We start by discussing some preliminaries in the next section. These include basic facts about polynomial rings and ideals, monomial orderings and zeros of polynomials as well as definitions of Furstenberg algebras, which are the main object we will work with in the paper. Next, in section \ref{sec:EEreduction} we discuss the  Ellenberg-Erman reduction from Theorem~\ref{SetTheorem} to a simpler statement involving only hyperplanes. Section \ref{sec:VarOfRichHyper} constructs the ideal which vanishes on rich hyperplane equations, the central object used to analyse the problem in the hyperplane case. Sections \ref{sec:IdealZero} and \ref{sec:MultCase} discuss the cases when this ideal is zero or not respectively. As at this point as we will have all the required tools, we prove Theorem \ref{thm:hyperBound} in section \ref{sec:higherdegree} by reducing it to the Furstenberg Algebra theorem for hyperplanes.

\section{Preliminaries}
\subsection{Polynomial rings, Ideals, and Varieties}\label{prelim:RingsIdeals}

The core algebraic objects we will be using are polynomial rings $\F[x_1,\hdots,x_n]$, their ideals $I$, and their quotients $\F[x_1,\hdots,x_n]/I$ for an arbitrary field $\F$. Later we will focus on $\F_q$ and its algebraic closure. Given a ring $R$, we use $\langle f_1,\hdots,f_k\rangle$ to refer to the ideal generated by elements $f_1,\hdots,f_k\in R$. The sum $I+J$ refers to the ideal generated by elements of the form $f+g$ with $f\in I,g\in J$. The product $IJ$ refers to the ideal generated by elements of the form $fg$ with $f\in I,g\in J$. It is easy to check that given two ideals $I$ and $J$, $(J+I)/I$ is an ideal of the ring $R/I$. It is also easy to see $(R/I)/((J+I)/I)=R/(I+J)$. For brevity, we write $(R/I)/((J+I)/I)$ as $(R/I)/J$. There is no cause for confusion as $(J+I)/I$ is precisely the ideal generated by $J$ in the ring $R/I$.

We recall the polynomial ring $\F[x_1,\hdots,x_n]$ is Noetherian. This means for any ideal $I$ of $\F[x_1,\hdots,x_n]$ we can find finitely many polynomials $f_1,\hdots,f_k\in I$ such that $I=\langle f_1,\hdots,f_k\rangle=\langle f_1\rangle +\hdots+\langle f_k\rangle$. Given an ideal $I$ of $\F[x_1,\hdots,x_n]$, the set $\V_\F(I)$ defined by $I$ is the subset of $\F^n$ on which all polynomials in $I$ vanish. This set may be empty. Given a finite set of points $S$ we define $\I_\F(S)$ as the ideal of polynomials in $\F[x_1,\hdots,x_n]$ which vanish on $S$. We write $\V_\F(\langle f_1,\hdots,f_k\rangle)$ as $\V_\F(f_1,\hdots,f_k)$. It is easy to check that $\V_\F(I+J)=\V_\F(I)\cap \V_\F(J)$. 

A polynomial $f\in \F[x_1,\hdots,x_n]$ of degree $d$ is said to be {\em homogenous} if it only consists of degree $d$ monomials. An ideal $I$ of $\F[x_1,\hdots,x_n]$ is said to be {\em homogenous} if it can be generated by a set of homogenous polynomials. An ideal $I$ of $\F[x_1,\hdots,x_n]$ is said to be {\em monomial} if it can be generated by a set of monomials. We note, if a polynomial $f$ belongs to a monomial ideal $I$ then all its monomials also belong to $I$.
\subsection{$\F$-algebras}\label{prelim:Falgebras}

\begin{define}[$\F$-algebras] A {\em finitely generated $\F$-algebra} $R$, is a ring $R$ of the form $\F[x_1,\hdots,x_n]/I$ where $I$ is an ideal of $\F[x_1,\hdots,x_n]$. We will omit the words ``finitely generated" from now on as all our algebras will be finitely generated.
\end{define}

For a $\F$-vector space $V$, we use $\text{dim}_\F V$ to represent its dimension. Finite dimensional (as a $\F$-vector space) $\F$-algebras can be used to capture certain geometric properties of a finite set of points in $\F^n$. Note, from now on whenever we talk about the dimension of an $\F$-algebra we mean its vector space dimension and not its Krull dimension (the dimension of the corresponding variety, which is always zero in our setting).

\begin{define}[Algebras from Point sets]
Given a finite set $S\subseteq \F^n$, we define $\text{Alg}(S)$ to be the $\F$-algebra $\F[x_1,\hdots,x_n]/\I_\F(S)$.
\end{define}



In dimension $1$ the picture is simple. For example take the point set $S=\{0,1\}\subseteq \F$. We see any polynomial in $\F[x]$ which vanishes on $S$ belongs to the ideal $\langle x(x-1)\rangle$. Therefore, $\text{Alg}(S)=\F[x]/\langle x(x-1)\rangle$. Evaluating polynomials at $0$ and $1$ produces an isomorphism of vector spaces from $\F[x]$ to  $\F^2$. This shows $\text{Alg}(S)$ is of dimension $2$. Take a polynomial $f(x)$ such that $f(0)=0$ and $f(1)\ne 0$. Using the isomorphism induced by the evaluation map $\F[x]\rightarrow \F^2$ we have $\text{Alg}(S)/\langle f(x)\rangle$ is isomorphic to $\F$ via the map which evaluates polynomials at $x=0$. This shows $\text{Alg}(S)/\langle f(x)\rangle=\text{Alg}(S\cap \V(f))=\text{Alg}(\{0\})$. Fortunately, this picture holds true in general.

\begin{prop}[Geometry of Algebras]\label{point2algebra}
Given a finite set $S$ of points in $\F^n$, the $\F$-algebra, $\text{Alg}(S)=\F[x_1,\hdots,x_n]/\I_\F(S)$ satisfies the following properties:
\begin{enumerate}
    \item $\text{Alg}(S)$ is an $\F$-vector space of dimension $|S|$, that is $\text{dim}_\F \text{Alg}(S)=|S|$.
    \item For an ideal $J\subseteq \F[x_1,\hdots,x_n]$, $\text{Alg}(S)/J$ equals $\text{Alg}(S\cap \V_\F(J))$ and hence is of dimension $|S\cap \V_\F(J)|$.
\end{enumerate}
\end{prop}
\begin{proof}
We write $\I_\F(S)$ as $I$ in this proof. For every point $b\in S$ we can define the map $\text{Eval}_b:\F[x_1,\hdots,x_n]\rightarrow \F$ which simply evaluates a polynomial at the point $b$. This map is linear. This map is also the same as the quotient map $\text{Eval}_b:\F[x_1,\hdots,x_n]\rightarrow \F[X]/\langle X-b\rangle$ where $X=(x_1,\hdots,x_n)$ and $\langle X-b\rangle=\langle x_1-b_1,\hdots,x_n-b_n\rangle$. 

Combining the $|S|$ evaluation maps together for each point in $S$ we have the linear map $\text{Eval}_S:\F[x_1,\hdots,x_n]\rightarrow \bigoplus_{b\in S} \F[X]/\langle X-b\rangle \cong  \F^{|S|}$. $R=\bigoplus_{b\in S} \F[X]/\langle X-b\rangle$ is a ring with a unit. It is easy to check that $\text{Eval}_S(fg)=\text{Eval}_S(f)\text{Eval}_S(g)$ and $\text{Eval}_S(1)=1$. This shows the map is a ring homomorphism. The kernel of this map is precisely going to be the ideal $I$ of polynomials vanishing on $S$. This means $\text{Eval}_S$ factors through an injective map $\phi$ from $\text{Alg}(S)=\F[x_1,\hdots,x_n]/I$ to $R\cong \F^{|S|}$ and the quotient map $\F[x_1,\hdots,x_n]\rightarrow \text{Alg}(S)$. This proves that $\text{Alg}(S)$ is finite dimensional.\\
\begin{center}\begin{tikzcd}
\F[x_1,\hdots,x_n]\arrow[rd]\arrow[r, "\text{Eval}_S"] & R=\bigoplus_{b\in S} \F[X]/\langle X-b\rangle \cong \F^{|S|} \\
& \text{Alg}(S) \arrow[u, "\phi"]
\end{tikzcd}\end{center}
By construction $\phi$ maps any element of $\text{Alg}(S)$ to its evaluation over $S$.

For all points $a\in S$ if we can find polynomials $f_a$ which vanish on $|S|\setminus\{a\}$ but not on $a$ using interpolation. This implies $\phi$ is a surjective map. This is the case because $f_a$ will map to a basis of $\F^{|S|}$. This would prove $\text{Alg}(S)$ is isomorphic to $\F^{|S|}$ as a vector space, via the map $\phi$. Hence, $\text{Alg}(S)$ is $|S|$ dimensional. 



We claim $\phi(J+I)$ will correspond exactly to $\bigoplus_{a\in S\setminus \V_\F(J)}\F[X]/\langle X-a\rangle$. As each polynomial in $J$ vanishes on $\V_\F(J)$ it follows $\phi(J+I)\subseteq \bigoplus_{a\in S\setminus \V_\F(J)}\F[X]/\langle X-a\rangle$. For each $a\in S\setminus \V_\F(J)$, we can find a polynomial $f_a$ which vanishes on $S\setminus \{a\}$ and $f_a(a)=1$, using interpolation. We can also find a polynomial $g_a\in J$ such that it does not vanish on $a$. $f_ag_a$ is then an element of $J$ and the span of $\phi(f_ag_a)$ is precisely $\bigoplus_{a\in S\setminus \V_\F(J)}\F[X]/\langle X-a\rangle$. This implies $\phi$ induces an isomorphism between $\text{Alg}(S)/J=\text{Alg}(S)/((J+I)/I)$ and $\bigoplus_{a\in S\cap \V_\F(J)}\F[X]/\langle X-a\rangle$ via evaluation of polynomials on the set $S\cap \V_\F(J)$. This proves $\text{Alg}(S)/J$ is isomorphic to $\text{Alg}(S\cap \V_\F(J))$.
\end{proof}

The previous proposition shows that for a finite set $S$, the finite dimensional $\F$-algebra $\text{Alg}(S)$ captures a number of geometric properties of $S$. In particular, the size of $S$ and the size of its intersections with varieties is captured. Not all finite dimensional $\F$-algebras need to be produced from a finite set of points like in the previous proposition. For example, the ring $\F[x_1,\hdots,x_n]/\langle x_1,\hdots,x_n\rangle^d,d>1$ is a finite dimensional algebra but contains lots of non-zero nillpotent elements while it is easy to check that an algebra produced by a finite set will have none. As $d$ varies we get different rings but all of them only have one maximal ideal $\langle x_1,\hdots,x_n\rangle$. Because we particularly care about subspaces we make the following definition.

\begin{define}[Algebra-subspace intersection]
Given an affine subspace $W\subseteq \F^n$ and a finite dimensional $\F$-algebra $R=\F[x_1,\hdots,x_n]/I$, we define {\em $R\sqcap W$} to be $R/\I_\F(W)$.
\end{define}

The ideal $\I_\F(W)$ is generated by any set of $n-k$ degree $1$ equations of hyperplanes in $\F^n$ whose intersection defines the $k$-dimensional affine subspace $W$. Now we can generalize the Furstenberg problem to this setting. 

\begin{define}[$(R,m)$-rich subspaces]
Given a finite dimensional $\F$-algebra $R=\F[x_1,\hdots,x_n]/I$ an affine subspace $W$ is said to be {\em $(R,m)$-rich} if $\text{dim}_\F R\sqcap W \ge m$.
\end{define}

\begin{define}[Furstenberg Algebras]\label{def:furstenbergAlgebra}
A finite dimensional $\F$-algebra $R=\F[x_1,\hdots,x_n]/I$ is said to be {\em $(k,m)$-Furstenberg}, if for any $k$-dimensional linear subspace $V$, some translate of $V$ is $(R,m)$-rich.
\end{define}

This definition is useful because of the following simple corollary of Proposition \ref{point2algebra}.

\begin{cor}[Furstenberg Sets to Algebras]\label{FrustSet2Alg}
Given a finite $(k,m)$-Furstenberg set $S\subseteq \F^n$, $\text{Alg}(S)$ is a $\F$-vector space of dimension $|S|$ and a $(k,m)$-Furstenberg Algebra.
\end{cor}

The previous corollary immediately shows that a lower bound for the dimension of Furstenberg Algebras can be lifted to produce a lower bound for the size of Furstenberg Sets.

We make similar definitions for the case of Hyper-Furstenberg sets.

\begin{define}[Hyper-Furstenberg Algebras]
Given a finite dimensional $\F$-algebra $R=\F[x_1,\hdots,x_n]/I$, a hypersurface $V_f$ defined as the zero of a polynomial $f\in \F[x_1,\hdots,x_n]$ is said to be {\em $(R,m)$-rich}, if $\text{dim}_\F \left(R/\langle f\rangle \right) \ge m$.

A finite dimensional $\F$-algebra $R=\F[x_1,\hdots,x_n]/I$ is said to be {\em $(m,d)$-Hyper-Furstenberg}, if for any hyperplane equation $h(x)=h_1x_1+\hdots+h_nx_n$ we can find a degree $d-1$ polynomial $g\in \F[x_1,\hdots,x_n]$ such that $h^d+g$ is $(R,m)$-rich.
\end{define}

Proposition \ref{point2algebra} immediately implies the following Corollary.

\begin{cor}[Hyper-Furstenberg Sets to Algebras]\label{cor:hyperFurstSet2alg}
Given a finite $(m,d)$-Hyper-Furstenberg set $S\subseteq \F^n$, $\text{Alg}(S)$ is a $\F$-vector space of dimension $|S|$ and a $(m,d)$-Hyper-Furstenberg Algebra.
\end{cor}
\subsection{Graded Lexicographic order and the basis of standard monomials}\label{prelim:monoOrder}
To better understand a finite dimensional $\F$-algebra we would like a nice basis for it. We will now construct one using the graded lexicographic order over monomials. The arguments here and in the next section are part of a more general treatment of monomial orders which can be found in Chapter 15 of \cite{eisenbud1995commutative}.

Let $\Z_{\ge 0}^n$ be the set of lattice points in $n$-dimensional space with non-negative coordinates. For $i\in \Z_{\ge 0}^n$ we let $\text{wt}(i)=\sum_t i_t$ be the {\em weight} of $i$. A monomial $f$ over variables $x=(x_1,\hdots,x_n)$ can be equivalently represented by an element of $\lambda\in\Z_{\ge 0}^n$ by writing $f=x^\lambda=x_1^{\lambda_1}\hdots x_n^{\lambda_n}$. The weight of $\lambda$ is precisely the degree of $f$.

\begin{define}[Graded Lexicographic order]
The {\em graded lexicographic order} $<$ (abbreviated as grlex) is a total order over monomials in the variables $x_1,\hdots,x_n$. For two monomials $f=x^{\lambda}$ and $g=x^\mu$, $f<g$ if $\text{wt}(\lambda)<\text{wt}(\mu)$ or $\text{wt}(\lambda)=\text{wt}(\mu)$ and $\lambda_i<\mu_i$ for the first index $i$ with $\lambda_i\ne \mu_i$.
\end{define}

We state a few properties of the grlex order which we will use. The grlex order satisfies $x_1>x_2>\hdots>x_n$. It also refines the partial order induced by divisibility. In other words, given two monomials $f_1$ and $f_2$ such that $f_1$ divides $f_2$ we have $f_1<f_2$. Finally, we note grlex is a well ordering. In other words, any non-empty set of monomials will have a least element under this order.
  
A monomial multiplied by a scalar is called a {\em term}. We can use the grlex order to compare terms in $\F[x_1,\hdots,x_n]$ by ignoring scalars. Given any polynomial $f\in \F[x_1,\hdots,x_n]$, we define the {\em initial term} of $f$ written as $\text{in}(f)$ as the largest term which is a part of $f$ (this will be the largest monomial and its corresponding scalar). For a set $X$ of polynomials, we let $\text{in}(X)$ be the set of initial terms of the polynomials in $X$. Given an ideal $I$, we let $\text{in}(I)$ be the ideal generated by initial terms of the polynomials in $I$. As each of the initial terms is a monomial, we see that $\text{in}(I)$ is a monomial ideal. It is called the {\em initial ideal} of $I$. Given an $\F$-algebra $R=\F[x_1,\hdots,x_n]/I$ we let $\text{in}(R)=\F[x_1,\hdots,x_n]/\text{in}(I)$. We now construct the special basis we need.

\begin{define}[Standard Monomials]\label{standardMono}
Given an $\F$-algebra $R=\F[x_1,\hdots,x_n]/I$, the set of monomials of $\F[x_1,\hdots,x_n]$ {\bf not in} $\text{in}(I)$ are called the {\em standard monomials} of $R$ and written as $\text{Std}(R)$. 
\end{define}

\begin{thm}[Monomial basis]\label{ThMono}
Given a finite dimensional $\F$-algebra $R=\F[x_1,\hdots,x_n]/I$, the standard monomials $\text{Std}(R)$ of $R$ form a basis for $R$ as an $\F$-vector space.
\end{thm}
\begin{proof}
Take monomials $g_1,g_2,..,g_k\in \text{Std}(R)$. We claim they are linearly independent in $R$. If they were linearly dependent then we could find $a_i\in \F$ such that $\sum_{i=1}^k a_ig_i=f \in I$. Hence, $\text{in}(f)\in \text{in}(I)$. As $\text{in}(f)$ will have to be one of $a_ig_i$ we obtain a contradiction.

Now suppose $\text{Std}(R)$ does not form a basis for $\F[x_1,\hdots,x_n]/I$. Consider the set of polynomials $X$ in $\F[x_1,\hdots,x_n]/I$ not spanned by $\text{Std}(R)$. We can pick the smallest term $h$ from $\text{in}(X)$. Pick a polynomial $f\in X$ such that $\text{in}(f)=h$. If $\text{in}(f)$ was in $\text{in}(I)$ we could find a polynomial $g\in I$ such that $\text{in}(f-g)<\text{in}(f)$. In $R$, $f-g$ is the same as $f$ and hence in $X$ but this would contradict the fact that $\text{in}(f)=h$ is the smallest term in $\text{in}(X)$. If $\text{in}(f)$ is not in $\text{in}(I)$ then we could find a monomial $m\in \text{Std}(R)$ and a scalar $a$ such that $\text{in}(f-ma)<\text{in}(f)$. Again, $f-ma\in X$ because $f$ is not spanned by the standard basis. We now have a contradiction as $\text{in}(f)$ is again not the smallest term in $\text{in}(X)$.
\end{proof}

\begin{cor}\label{InDimPres}
Given a finite dimensional $\F$-algebra $R=\F[x_1,\hdots,x_n]/I$, we have that,
$$\text{dim}_\F R=\text{dim}_\F \text{in}(R).$$
\end{cor}
\begin{proof}
As $\text{Std}(R)=\text{Std}(\text{in}(R))$, Theorem \ref{ThMono} shows that the same set of monomials form a basis for $R$ and $\text{in}(R)$.
\end{proof}

Using the initial ideal operation we get a nice algebra of the same dimension. We can preserve even more information by means of a different operation. Any polynomial $f$ can be written as $\sum_{d=0}^D f_d$ where $f_d$ is a homogenous polynomial of degree $d$, where $D$ is the degree of $f$. We let $\text{hd}(f)$ refer to $f_D$. For an ideal $I$ we also define $\text{hd}(I)$ as the ideal generated by $\text{hd}(f)$ for all $f\in I$. We see that $\text{hd}(I)$ is homogenous. Given an $\F$-algebra $R=\F[x_1,\hdots,x_n]/I$, we let $\text{hd}(R)=\F[x_1,\hdots,x_n]/\text{hd}(I)$. To prove properties about $\text{hd}(I)$ we use $\text{in}(I)$ and the following lemma connecting the two.

\begin{lem}\label{connectHDandIN}
Given an ideal $I$ of $\F[x_1,\hdots,x_n]$ we have,
$$\text{in}(\text{hd}(I))=\text{in}(I).$$
\end{lem}
\begin{proof}
In one direction, for any $f\in I$ we have $\text{hd}(f)\in \text{hd}(I)$ and $\text{in}(f)=\text{in}(\text{hd}(f))$. This implies $\text{in}(I)\subseteq \text{in}(\text{hd}(I))$. 

Now we prove the other inclusion. Take $g\in \text{hd}(I)$. It is of the form $\sum_i h_i\text{hd}(f_i)$ where $f_i\in I$ and $h_i$ are homogenous polynomials. As $\text{hd}(f_i)$ are homogenous and $h_i$ are homogenous, $h_i\text{hd}(f_i)$ is also homogenous. The sum $\sum_i h_i\text{hd}(f_i)$ can then be split into parts with the same degree. This means $\text{hd}(g)$ will be of the form $\sum_i h'_i\text{hd}(f'_{i})$ with $h'_i$ homogenous and $f'_i\in I$. In fact we have,
$$hd(g)=\sum_i h'_i\text{hd}(f_i)=\sum_i \text{hd}(h'_if_i)= \text{hd}\left(\sum_i h'_if_i\right)$$ 
as $h'_i$ are homogenous. We finally note,
$$\text{in}(g)=\text{in}(\text{hd}(g))=\text{in}\left(\text{hd}\left(\sum_i h'_if_i\right)\right)=\text{in}\left(\sum_i h'_if_i\right)$$
But $\sum_i h'_if_i\in I$ which implies $\text{in}(I)\supseteq \text{in}(\text{hd}(I))$. 
\end{proof}
We now prove a lemma proved in \cite{ellenberg2016furstenberg} using alternate elementary arguments. The original proof uses algebraic geometric arguments and properties of flat families.
\begin{lem}\label{dilationLem}
Given a finite dimensional $\F$-algebra, $R=\F[x_1,\hdots,x_n]/I$, we have the following:
\begin{enumerate}
    \item $\text{dim}_\F R= \text{dim}_\F \text{hd}(R)$.
    \item For an ideal $J$, we have $\text{dim}_\F \left( R/J \right) \le \text{dim}_\F \left( \text{hd}(R)/\text{hd}(J) \right).$
\end{enumerate}
\end{lem}
\begin{proof}
Using Corollary \ref{InDimPres} and Lemma \ref{connectHDandIN}, we have 

\begin{align*}
    \text{dim}_\F R&=\text{dim}_\F \text{in}(R)=\text{dim}_\F \left(\F[x_1,\hdots,x_n]/\text{in}(I)\right)\\
    &=\text{dim}_\F\left(\F[x_1,\hdots,x_n]/\text{in}(\text{hd}(I))\right)=\text{dim}_\F\text{in}(\text{hd}(R))=\text{dim}_\F\text{hd}(R).
\end{align*}
Recall for ideals $K_1$ and $K_2$, $(R/K_1)/K_2=R/(K_1+K_2)$. Given an ideal $J$, using the first claim of this lemma we have, 
\begin{align*}
    \text{dim}_\F \left( R/J\right)&=\text{dim}_\F \left(\F[x_1,\hdots,x_n]/(I+J)\right)\\ &=\text{dim}_\F \left(\F[x_1,\hdots,x_n]/\text{hd}(I+J)\right)\\
    &\le \text{dim}_\F \left(\F[x_1,\hdots,x_n]/(\text{hd}(I)+\text{hd}(J))\right)\\
    &=\text{dim}_\F \left(\text{hd}(R)/\text{hd}(J)\right).
\end{align*} 
The inequality follows from the fact that $(\text{hd}(I)+\text{hd}(J))\subseteq \text{hd}(I+J)$.
\end{proof}

\subsection{Generic Initial Ideals}\label{prelim:genericInitial}
The theorem in this section will only be used in Section \ref{sec:IdealZero}. Every ideal has a ``canonical" initial ideal associated with it which is invariant under the action of the Borel group, that is the group of upper triangular invertible matrices. In this section we will make this statement precise. First, we need to define the Borel group and its action on polynomials.

The {\em Borel group} $\B(n,\F)$ is the group of $n\times n$ upper-triangular invertible matrices over the field $\F$. Given an element $g\in \B(n,\F)$, we define its action over a polynomial $f\in \F[x_1,\hdots,x_n]$ as
$$\leftidx{^g}{{\!f}}{}(x)=f(xg),$$
where $xg$ is the product of the matrix $g$ with the row vector $x=(x_1,\hdots,x_n)$. Given an ideal $I$, $\leftidx{^g}{{\!I}}{}$ refers to the ideal generated by $\leftidx{^g}{{\!f}}{}$ for all polynomials $f\in I$. $\B(n,\F)$ can be identified with a subset of $\F^{n(n-1)/2}$ described using $n(n-1)/2$ indeterminates $b_{ij},1\le i\le j\le n$ corresponding to the non-zero entries in the upper triangular matrix. 
The following theorem is a standard result from commutative algebra (see e.g., Chapter 15 in  \cite{eisenbud1995commutative}). For the sake of completeness we include a somewhat simplified proof in  Appendix~\ref{apx:GIN}.

\begin{thm}[Generic Initial Ideals]\label{genericInitialThm}
Given an infinite field $\F$ and $I$ a homogenous ideal of $\F[x_1,\hdots,x_n]$, there exists a monomial ideal $\text{GIN}(I)$ called the {\em generic initial ideal} of $I$ with the following properties:
\begin{enumerate}
    \item There exists a non-zero polynomial $q$ in the indeterminates $b_{ij},1\le i\le j\le n$, such that for any $g\in \B(n,\F)$ for which $q(g)\ne 0$ we have $\text{in}(\leftidx{^g}{{\!I}}{})=\text{GIN}(I)$.
    \item The ideal $\text{GIN}(I)$ is stable under the action of the Borel group. That is, given any element $g\in \B(n,\F)$  we have $\leftidx{^g}{{\!\text{GIN}(I)}}{}=\text{GIN}(I)$.
\end{enumerate}
\end{thm}

\subsection{Method of multiplicities}\label{prelim:multiplicities}

The results here are from a paper by Dvir, Kopparty, Saraf, and Sudan~\cite{dvir2013extensions}. We state the theorems we need and the proofs can be found in the aforementioned paper.

\begin{define}[Hasse Derivatives]
Given a polynomial $f\in \F[x_1,\hdots,x_n]$ and a $i\in \Z_{\ge 0}^n$ the $i$th {\em Hasse derivative} of $f$ is the polynomial $f^{(i)}$ in the expansion $f(x+z)=\sum_{i\in \Z_{\ge 0}^n} f^{(i)}(x)z^i$ where $x=(x_1,...,x_n)$, $z=(z_1,...,z_n)$ and $z^i=\prod z_j^{i_j}$.  
\end{define}

They satisfy some useful identities. We state two simple ones that we will use.

\begin{lem}\label{derivativeRule}
Given polynomials $f,g\in \F[x_1,\hdots,x_n]$ and $i\in \Z_{\ge 0}^n$ we have, 
$$(f+g)^{(i)}=f^{(i)}+g^{(i)}\text{ and }(fg)^{(i)}=\sum_{j+k=i} f^{(j)}g^{(k)}.$$
\end{lem}

We make precise what it means for a polynomial to vanish on a point $a\in \F^n$ with multiplicity. First we recall for a point $j$ in the non-negative lattice $\Z^n_{\ge 0}$, its weight is defined as $\text{wt}(i)=\sum_{i=1}^n j_i$.

\begin{define}[Multiplicity]
For a polynomial $f$ and a point $a$ we say $f$ vanishes on $a$ with {\em multiplicity} $N$, if $N$ is the largest integer such that all Hasse derivatives of $f$ of weight strictly less than $N$ vanish on $a$. We use $\text{mult}(f,a)$ to refer to the multiplicity of $f$ at $a$.
\end{define}

Notice, $\text{mult}(f,a)=1$ just means $f(a)=0$. We will use the following simple property concerning multiplicities of composition of polynomials.

\begin{lem}\label{lem:multComp}
Given a polynomial $f\in \F[x_1,\hdots,x_n]$ and a tuple $g=(g_1,\hdots,g_n)$ of polynomials in $\F[y_1,\hdots,y_m]$, and $a\in \F^m$ we have, 
$$\text{mult}(f\circ g, a)\ge \text{mult}(f,g(a)).$$
\end{lem}

The key lemma here is an extended Schwartz-Zippel bound~\cite{schwartz1979probabilistic}\cite{ZippelPaper} which leverages multiplicities and is proven in \cite{dvir2013extensions}.

\begin{lem}[Schwartz-Zippel with multiplicity]\label{multSchwartz}
Let $f\in \F[x_1,..,x_n]$, with $\F$ an arbitrary field, be a nonzero polynomial of degree at most $d$. Then for any finite subset $U\subseteq \F$ ,
$$\sum\limits_{a\in U^n} \text{mult}(f,a) \le d|U|^{n-1}.$$
\end{lem}
\section{The Ellenberg-Erman reduction}\label{sec:EEreduction}
We are going to make a series of reductions starting from Theorem \ref{SetTheorem} to end up at a simpler problem just involving hyperplanes. The first step in the reduction is provided by Corollary \ref{FrustSet2Alg} which shows Furstenberg Sets $S$ produce $|S|$-dimensional Furstenberg Algebras. One of the trickier aspects about dealing with Furstenberg sets and Algebras arises from the translating of subspaces. We can perform a ``dilation" operation on Furstenberg Algebras to move all rich subspaces to the origin. First, we make a definition.

\begin{define}[Homogenous Furstenberg Algebra]
A finite dimensional $\F$-algebra $R=\F[x_1,\hdots,x_n]/I$ with $I$ homogenous is said to be {\em $(k,m)$-Hom-Furstenberg}, if all $k$-dimensional {\bf linear subspaces} $V$ in $\F^n$ are $(S,m)$-rich. 
\end{define}

The next lemma shows Furstenberg Algebras can be transformed to homogenous Furstenberg Algebras.

\begin{lem}[Reduction to Hom Furstenberg]\label{DFrustLem}
Given a finite dimensional $(k,m)$-Furstenberg Algebra $R=\F[x_1,\hdots,x_n]/I$, $\text{hd}(R)=\F[x_1,\hdots,x_n]/\text{hd}(I)$ is a $(k,m)$-Hom-Furstenberg Algebra such that $\text{dim}_\F \text{hd}(R)=\text{dim}_\F R$.
\end{lem}
\begin{proof}
Using Lemma \ref{dilationLem} we have that $\text{dim}_\F \text{hd}(R)=\text{dim}_\F R$. We claim all $k$-dimensional linear subspaces $V$ are $(\text{hd}(R),m)$-rich. This is because given some $V$ a translate of it $V'$ will be $(R,m)$-rich. Then we note the ideals $\I_\F(V)$ and $\I_\F(V')$ are generated by the equations of hyperlanes containing $V$ and $V'$ respectively. This implies $\I_\F(V)=\text{hd}(\I_\F(V))=\text{hd}(\I_\F(V'))$. The second claim in Lemma \ref{dilationLem} implies 
$$ m \le \text{dim}_\F \left( R/\I_\F(V') \right) \le \text{dim}_\F \left( \text{hd}(R)/\text{hd}(\I_\F(V'))\right)= \text{dim}_\F \left(\text{hd}(R)/\I_\F(V)\right).$$
This shows $V$ is $(\text{hd}(R),m)$-rich.
\end{proof}

Ellenberg and Erman in \cite{ellenberg2016furstenberg} call this step {\em dilation}. This terminology is most clear when we think of $\F=\R$. The process of taking the highest degree term of a $d$-degree polynomial $f$ can be thought of as taking the limit $t^{d}f(x/t)$ as $t\rightarrow \infty$ which corresponds to dilating the zero set of $f$ towards the origin. We have the following bound on Homogenous Furstenberg Algebras.

\begin{thm}[Homogenous Furstenberg Algebra Bound]\label{AlgebraBound}
A $(k,m)$-Hom-Furstenberg Algebra $R=\F_q[x_1,\hdots,x_n]/I$ with $m\le q^k$ satisfies the bound,
$$\text{dim}_{\F_q} R\ge C_{n,k}m^{n/k}$$
where $C_{n,k}=\Omega(1/16^{n\ln(n/k)})$.
\end{thm}

Ellenberg and Erman in \cite{ellenberg2016furstenberg} show that Theorem \ref{AlgebraBound} is tight in the  exponent of $m$ using the algebra $\F[x_1,\hdots,x_n]/\langle x_1,\hdots,x_n\rangle^d$. They also produce an example to show why the condition $m\le q^k$ is necessary (at least for $k=1$). We discuss this example below.

\begin{example}[$m\le q^k$ is necessary]\label{ex:mBoundReq}
Consider $R_N=\F_q[x_1,x_2]/I$ with 
$$I=\langle x_1,x_2^{q^N}\rangle\prod\limits_{r\in \F_q} \langle x_2+rx_1,x_1^{q^N}\rangle.$$
It is not hard to check for a line $L_t$ with equation $x_2+tx_1=0$ for $t\in \F_q$, $L\sqcap R_N$ has dimension at least $q^N$. For the line $L_\infty$ with equation $x_1=0$ the same holds. This shows $R_N$ is $(1,q^N)$-Hom-Furstenberg. If Theorem \ref{AlgebraBound} did not have the condition $m\le q^k$, then it would imply that $\text{dim}_{\F_q} R_N \ge C_{n,k} q^{2N}$. 

Consider the map,
$$R_N\rightarrow \frac{\F_q[x_1,x_2]}{\langle x_1,x_2^{q^N}\rangle}\oplus\bigoplus\limits_{r\in \F_q} \frac{\F_q[x_1,x_2]}{\langle x_2+rx_1,x_1^{q^N}\rangle}$$
obtained by combining the quotient maps $R_N\rightarrow \F_q[x_1,x_2]/\langle x_2+rx_1,x_1^{q^N}\rangle$ for all $r\in F_q$ and $R_N \rightarrow \F_q[x_1,x_2]/\langle x_1,x_2^{q^N}\rangle$. It is easy to check this map is injective. This implies $\text{dim}_{\F_q} R_N\le q^{N}(q+1)$ which leads to a contradiction for $N\ge 2$. Hence, the condition $m\le q^k$ is necessary.
\end{example}

We can prove Theorem \ref{SetTheorem} given Theorem \ref{AlgebraBound}.

\begin{proof}[Proof of Theorem \ref{SetTheorem}]
Given a $(k,m)$-Furstenberg set $S$ in $\F^n_q$, $\text{Alg}(S)$ is a $(k,m)$-Furstenberg Algebra of dimension $|S|$ using Corollary \ref{FrustSet2Alg}. Lemma \ref{DFrustLem} shows $\text{hd}(\text{Alg}(S))$ is $(k,m)$-Hom-Furstenberg Algebra of dimension $|S|$. Theorem \ref{AlgebraBound} now proves Theorem \ref{SetTheorem}.
\end{proof}

Because we do not have to worry about translations anymore it suffices to prove Theorem \ref{AlgebraBound} for the case $k=n-1$ and using an induction argument. In the case $k=n-1$ we are taking intersection with hyperplanes containing the origin. We first state the hyperplane version of Theorem \ref{AlgebraBound}.

Theorem \ref{thm:furstPlaneBound} and Lemma \ref{recurseBound} proves $C_{n,k}=\Omega(1/16^{n\log(n/k)})$. The constant in Theorem \ref{thm:furstPlaneBound} can be optimized but unfortunately not all the way to match the bound in Example \ref{constantBound} below.

\begin{thm}[Hyperplane Furstenberg Bound]\label{thm:furstPlaneBound}
An $(n-1,m)$-Hom-Furstenberg Algebra $R=\F_q[x_1,\hdots,x_n]/I$ with $m\le q^{n-1}$ satisfies the bound,
$$\text{dim}_{\F_q} R \ge C_{n,n-1}m^{n/(n-1)},$$
where $C_{n,n-1}\ge 1/16$.
\end{thm}
The following lemma shows how Theorem \ref{thm:furstPlaneBound} proves Theorem \ref{AlgebraBound}.

\begin{lem}\label{recurseBound}
Suppose Theorem \ref{thm:furstPlaneBound} is true for all $n$ and for some constant $C_{n,n-1}$ then for all $k,n$ a $(k,m)$-Hom-Furstenberg Algebra $R=\F_q[x_1,\hdots,x_n]/I$ satisfies the bound,
\begin{align*}
    \text{dim}_{\F_q} R \ge \prod\limits_{i=k+1}^{n}C_{i,i-1}^{n/i} m^{n/k}. 
\end{align*}

\end{lem}
\begin{proof}
Say $R=\F_q[x_1,\hdots,x_n]/I$ is a $(k,m)$-Hom-Furstenberg algebra. Given any $k+1$ dimensional linear subspace $V$ consider the algebra $R\sqcap V=R/\I_{\F_q}(V)=\F_q[x_1,\hdots,x_n]/(I+\I_{\F_q}(V))$. $\I_{\F_q}(V)=\langle h_1,\hdots,h_{n-k-1}\rangle$ where $h_{1},\hdots,h_{n-k-1}$ are $n-k-1$ linearly independent degree $1$ equations of hyperplanes containing $V$. By performing a change of basis in the field $\F_q$ on the space defined by $x_1,\hdots,x_n$ and renaming, we can assume $h_1,\hdots,h_{n-k-1}$ are the coordinate hyperplanes $x_1,\hdots,x_{n-k-1}$. That is, $V$ is the subspace defined by $x_1=\hdots=x_{n-k-1}=0$ after the base change. This means $R\sqcap V=R/\I_{\F_q}(V)=\F_q[x_{n-k},\hdots,x_n]/I'$ where $I'$ is the ideal generated by the restricted polynomials $f'(x_{n-k},\hdots,x_n)=f(0,0,\hdots,0,x_{n-k},\hdots,x_n)$ for every polynomial $f\in I$. 

Any $k$ dimensional linear subspace $W$ contained in $V$ corresponds to a hyperplane in $V$. In other words it will be a hyperplane in the variables $x_{n-k},\hdots,x_n$. We claim that $(R\sqcap V)\sqcap W=R/(\I_{\F_q}(V)+\I_{\F_q}(W))=R/\I_{\F_q}(W)=R\sqcap W$. This is the case because any hyperplane containing $V$ automatically contains $W$. We know $\text{dim}_{\F_q} R\sqcap W\ge m$. This implies $\text{dim}_{\F_q} (R\sqcap V)\sqcap W\ge m$. As $R\sqcap V=\F_q[x_{n-k},\hdots,x_n]/I'$ and $W$ is a hyperplane in $V$ which is spanned by $x_{n-k},\hdots,x_n$, we have $R\sqcap V$ is $(k,m)$-Furstenberg. As we are supposing Theorem \ref{AlgebraBound} is true for hyperplanes we have $\text{dim}_{\F_q} R\sqcap V \ge m^{1+1/k}C_{k+1,k}$.

Applying this argument recursively for a $(k,m)$-Hom-Furstenberg Algebra $R$ proves
\begin{align*}
    \text{dim}_{\F_q} R \ge \prod\limits_{i=k+1}^{n}C_{i,i-1}^{n/i} m^{n/k}. 
\end{align*}

One might worry that during the recursion $m$ is no longer an integer but that is not a problem because at every stage when we calculate the bound we can take ceiling of the bound to get a better bound. As the bound is increasing in $m$ we will not run into problems.
\end{proof}

Theorem \ref{thm:furstPlaneBound} and Lemma \ref{recurseBound} proves $C_{n,k}=\Omega(1/16^{n\log(n/k)})$ in Theorem \ref{AlgebraBound}. Ellenberg and Erman in \cite{ellenberg2016furstenberg} prove Theorem \ref{thm:furstPlaneBound} with $C_{n,n-1}=\Omega(1/n)$. Lemma \ref{recurseBound} gives them $C_{n,k}=1/n^{O(n\ln(n/k))}$ in Theorem \ref{AlgebraBound}. We will prove Thoerem \ref{AlgebraBound} with $C_{n,n-1}=1/16$. We therefore obtain the constant $C_{n,k}=\Omega(1/16^{n\ln(n/k)})$ for $k<n-1$. To obtain the Theorem \ref{KakeyaTheorem} bound we would need $C_{n,1}$ to be $1/2^n$. The reason we can not recover this bound is because this bound does not hold for Algebras as seen from the following example.

\begin{example}[Upper bound on $C_{n,k}$]\label{constantBound}
The finite dimensional $\F$-algebra $\F_q[x_1,\hdots,x_n]/\langle x_1,\hdots,x_n\rangle^d$ has dimension ${d-1+n \choose n}$ and is $(k,{d-1+k\choose k})$-Hom-Furstenberg. This gives us the following bounds for Theorem \ref{AlgebraBound}.
$$C_{n,k}\le \frac{{d-1+n \choose n}}{{d-1+k\choose k}^{n/k}}\le \left(\prod\limits_{j=k+1}^n \frac{d-1+j}{j}\right) {d-1+k\choose k}^{-(n-k)/k}\le \left(\prod\limits_{j=k+1}^n \frac{d-1+j}{j}\right) \frac{(k!)^{(n-k)/k}}{d^{k-n}} ,$$
for all $d\le q$. At the last step we used the inequality $(n-k+1)^k/k!\ge {n \choose k}$. We set $d=q$, and as the bound is field independent, we can let $q$ grow. This gives us the bound,
$$C_{n,k}\le \frac{(k!)^{n/k}}{n!}\le O(e^{-n\log(n/k)}).$$
In particular, for $C_{n,1}$ we have,
$$C_{n,1}\le \frac{1}{n!}=O(e^{-n\log(n)}). $$
For $C_{n,n-1}$ we have,
$$C_{n,n-1}\le \frac{(n-1)!^{1/(n-1)}}{n}=\frac{1}{e}+o_n(1),$$
where $o_n(1)$ tends to $0$ as $n$ grows towards infinity.
\end{example}



\section{The variety of $(R,m)$-rich hyperplanes}\label{sec:VarOfRichHyper}

To prove Theorem \ref{AlgebraBound} for hyperplanes we will treat the space of linear hyperplanes in $\F^n_q$ as an algebraic space and provide a recipe to construct the space of $(R,m)$-rich hyperplanes as a variety defined by an ideal. We will skip saying linear as from now on all our hyperplanes will contain the origin. Given a field $\F$, its extension $\E$, and an ideal $I$ of $\F[x_1,\hdots,x_n]$, we let $I^\E$ be the ideal of $\E[x_1,\hdots,x_n]$ generated by the polynomials in $I$. We extend this notation further so that given $R=\F[x_1,\hdots,x_n]/I$, $R^\E$ refers to the algebra $\E[x_1,\hdots,x_n]/I^\E$.

Given a hyperplane $V_h$ in $\F^n$ defined by the equation $h(x)=h_1x_1+...+h_{n}x_{n}$ with $(h_1,\hdots,h_n)\in \F^n$, we note that the coefficients $h_i$ provide variables which allow us to consider a general hyperplane. We can find polynomials in these variables which vanish at hyperplanes with some given property. The objective in this section is to prove the following theorem.

\begin{thm}[Variety of $(R,m)$-rich hyperplanes]\label{Jmdef}
Given a finite dimensional $\F$-algebra $R=\F[x_1,\hdots,x_n]/I$, where $\F$ is an {\bf arbitrary} field,  and some number $m\ge 0$, there exists an ideal $J_m(R)$ in the ring $\F[h_1,\hdots,h_n]$, where $h_1,\hdots,h_n$ are variables defining a general hyperplane equation $h(x)=h_1x_1+\hdots+h_nx_n$, such that the following properties are satisfied:
\begin{enumerate}
    \item $\V_\F(J_m(R))\subseteq \F^n$ is the set of $(R,m)$-rich hyperplane equations with coefficients in $\F$.
    \item $J_m(R)$ is either $\langle 0\rangle$ or generated by homogenous polynomials of degree $(\text{dim}_\F R)-m+1$.
    \item Moreover, given a field extension $\E$ of $\F$, we have $J_m(R^\E)=J_m(R)^{\E}$.
\end{enumerate}
\end{thm}
\begin{proof}
We want to understand the intersection of a finite dimensional $\F$-algebra $R=\F[x_1,\hdots,x_n]/I$ with one hyperplane $V_h$ described by the equation $h(x)=h_1x_1+\hdots+h_nx_n=0$. The intersection involves quotienting out the ideal generated by $h$. This ideal is precisely the image of the multiplication map $T_h:R\rightarrow R$ mapping an element $f$ to $hf$. If $V_h$ is $(R,m)$-rich, that is $\text{dim}_\F R\sqcap V_h \ge m$, then the image of $T_h$ is required to be of dimension at most $(\text{dim}_\F R)-m$. This is the case because quotienting $R$ by the image of $T_h$ produces $R\sqcap V_h$. In other words, we require the rank of $T_h$ to be strictly less than $(\text{dim}_\F R)-m+1$. This condition is the same as writing the matrix of the map $T_h$ in any basis and requiring that all minors of $T_h$ of size $(\text{dim}_\F R)-m+1\times (\text{dim}_\F R)-m+1$ vanish. Each of these minors will be a polynomial in $\F[h_1,\hdots,h_n]$. We generate an ideal $J_m(R)$ out of them. By construction, $\V_\F(J_m(R))$ is the set of $(R,m)$-rich hyperplanes equations with coefficients in $\F$.

To prove the second claim, we can show that these minors are $(\text{dim}_\F R)-m+1$ degree homogenous polynomials in any basis. To make our life easier we take the basis of standard monomials $\text{Std}(R)$. The size of the basis set is $\text{dim}_\F R$. For any monomial $f\in \text{Std}(R)$, $T_h(f)=(h_1x_1+\hdots+h_nx_n)f=h_1x_1f+\hdots+h_nx_nf$. Any $x_if$ which is not a standard monomial will be $0$ in $R$. Thus we see that the entries in the matrix of the map $T_h$ come from the set $\{0,h_1,\hdots,h_n\}$. This immediately implies that the $(\text{dim}_\F R)-m+1\times (\text{dim}_\F R)-m+1$ minors are homogenous polynomials of degree $(\text{dim}_\F R)-m+1$ if they are non-zero.

We prove the third claim, by noting that the standard basis of monomial for $R$ and $R^\E$ are the same. This means the matrix defining $T_h$ is same for both rings in this basis. This means the minors of this matrix and hence the generators of $J_m(R)$ and $J_m(R^\E)$ are the same.
\end{proof}

We will use Theorem \ref{Jmdef} on a $(n-1,m)$-Hom-Furstenberg Algebra $R=\F_q[x_1,\hdots,x_n]/I$ to obtain the ideal $J_m(R)$. This ideal can be $\langle 0 \rangle$. 

Consider the finite dimensional algebra 
$$P=\F_q[x_1,\hdots,x_n]/\langle x_1,\hdots,x_n\rangle^d,$$ 
considered in Example \ref{constantBound}. We note for $\overline{P}=P^{\kF_q}=\overline{\F_q}[x_1,\hdots,x_n]/\langle x_1,\hdots,x_n\rangle^d$, where $\overline{\F_q}$ is the algebraic closure of $\F_q$, $J_m(\overline{P})=\langle 0\rangle$ for $m={d+n-2\choose n-1}$. This is the case because every hyperplane in $\overline{\F_q}^n$ is $(\overline{P},m)$-rich. This means that each polynomial in $J_m(\overline{P})$ vanishes on every point in $\overline{\F_q}^n$, which implies $J_m(\overline{P})=\langle 0\rangle$. Finally, item 3 of Proposition \ref{Jmdef} shows $J_m(P)=\langle 0\rangle$.

When $J_{l}(R)=\langle 0\rangle$ for some $l\ge 0$ we have the following bound.

\begin{thm}[$J_l(R)=\langle 0\rangle$ case]\label{thm:furstPlaneZero}
Given an $(n-1,l)$-Hom-Furstenberg Algebra $R=\F_q[x_1,\hdots,x_n]/I$ such that $J_l(R)=\langle 0\rangle$ for some $l>0$ we have the following bound,
$$\text{dim}_{\F_q} R\ge 4^{-1}l^{n/(n-1)}.$$
    
\end{thm}


$R$ being Furstenberg does imply that $\V_{\F_q}(J_m(R))$ will contain $\F_q^n\setminus \{0\}$ but that doesn't mean $J_m(R)$ is always the $\langle 0\rangle$ ideal. For example, consider $Q=\F_2[x_1,x_2]/I$ with $$I=\langle x_2,x_1^8\rangle\langle x_1+x_2,x_1^8\rangle\langle x_1,x_2^8\rangle,$$ 
considered in Example \ref{ex:mBoundReq}. It is easy to check that $\F^2_2\subseteq \V_{\F_2}(J_{10}(Q))$. But $J_{10}(Q)$ can't be $\langle 0 \rangle$. If it were then for $Q^{\F_{4}}=\F_{4}[x_1,x_2]/I$ we will also have $J_{10}(Q^{\F_4})=\langle 0 \rangle$ using Proposition \ref{Jmdef}. $J_{10}(Q^{\F_4})$ can't be $\langle 0 \rangle$. Take $f=x_1+ax_2$ with $a\in \F_{4}\setminus \F_{2}$. It is easy to check $|Q^{\F_4}/\langle f\rangle|<10$. This means $J_{10}(Q)$ is not $\langle 0\rangle$.

We also prove the following bound corresponding to the case when $J_l(R)\ne \langle 0\rangle$.

\begin{thm}[$J_l(R)\ne \langle 0\rangle$ case]\label{thm:furstPlaneNon}
Given an $(n-1,m)$-Hom-Furstenberg Algebra $R=\F_q[x_1,\hdots,x_n]/I$ such that $J_l(R)\ne 0$ for some $l\le m$ we have the following bound,
$$\text{dim}_{\F_q} R \ge qm\left(1-\frac{l-1}{m}\right).$$
\end{thm}

Theorem \ref{thm:furstPlaneZero} and \ref{thm:furstPlaneNon} immediately imply  Theorem \ref{thm:furstPlaneBound}.

\begin{proof}[Proof of Theorem \ref{thm:furstPlaneBound}]
Given an $(n-1,m)$-Furstenberg Algebra $R$, if $J_{\lceil m/2\rceil}(R)=0$, Theorem \ref{thm:furstPlaneZero} implies,
$$\text{dim}_{\F_q} R \ge \frac{m^{n/(n-1)}}{2^{n/(n-1)}4}\ge \frac{m^{n/(n-1)}}{16}.$$
If $J_{\lceil m/2 \rceil}(R)\ne 0$ then Theorem \ref{thm:furstPlaneNon} implies,
$$\text{dim}_{\F_q} R \ge \frac{mq}{2}\ge \frac{m^{n/(n-1)}}{2},$$
where in the last step we used the fact that $m\le q^{n-1}$.
\end{proof}
\section{Theorem \ref{thm:furstPlaneZero} using Borel stable ideals}\label{sec:IdealZero}

This section will use the notations and definitions of Section~\ref{prelim:genericInitial}. We recall a monomial ideal $K$ of the ring $\F[x_1,\hdots,x_n]$ is said to be Borel stable if, for any element $g$ in the Borel group $\B(n,\F)$  we have $\leftidx{^g}{\!K}{}=K$. Let $\kF_q$ be the algebraic closure of $\F_q$.

We will prove Theorem \ref{thm:furstPlaneZero} by first showing that moving to the algebraic closure of $\F_q$ still produces a Furstenberg Algebra. At that point, we use the Generic Initial ideal construction from Theorem~\ref{genericInitialThm} to produce a Furstenberg Algebra whose basis of standard monomials satisfies a nice combinatorial property. Using that combinatorial property we prove the required bound.

\begin{lem}[Extending to $\kF_q$]\label{extendingLemma}
Given $R=\F_q[x_1,\hdots,x_n]/I$ with $I$ homogenous such that $J_m(R)=\langle 0\rangle$, we have that the $\kF_q$-algebra $\kR=\kF_q[x_1,\hdots,x_n]/I$ is $(n-1,m)$-Hom-Furstenberg and a $\kF_q$-vector space of dimension $\text{dim}_{\F_q} R$. In other words, every hyperplane with coefficients in $\kF_q$ is $(\kR,m)$-rich.
\end{lem}
\begin{proof}
Both $\kR=\kF_q[x_1,\hdots,x_n]/I$ and $R$ are spanned by monomials not in $\text{in}(I)$, that is the standard monomials of $R$ (or $\kR$). This shows $\text{dim}_{\kF_q} \kR=\text{dim}_{\F_q} R$.
By item three in Proposition \ref{Jmdef} we have $J_m(\kR)=J_m(R)^{\kF_q}=\langle 0 \rangle$. This means all hyperplanes in $\kF^n_q$ are $(\kR,m)$-rich.
\end{proof}

As $\kF_q$ is infinite we can use Theorem \ref{genericInitialThm}. Now, for the second step in our reduction.

\begin{lem}[Degeneration to Generic Initial ideal]\label{degenerationLemma}
Given an $(n-1,m)$-Hom-Furstenberg Algebra $R=\kF_q[x_1,\hdots,x_n]/I$, we define the $\kF_q$-Algebra $\hat{R}=\kF_q[x_1,\hdots,x_n]/K$ where $K=\text{GIN}(I)$ is the generic initial ideal of $I$. Then $\hat{R}$ is a finite dimensional $\kF_q$-algebra with dimension $\text{dim}_{\kF_q} R$ and is $(n-1,m)$-Hom-Furstenberg.
\end{lem}
\begin{proof}
From Theorem \ref{genericInitialThm} we know we can find a $g\in\B(n,\kF_q)$ such that $K=\text{in}(\leftidx{^g}{\!I}{})$. Using Corollary \ref{InDimPres} we have $\text{dim}_{\kF_q} \hat{R}=\text{dim}_{\kF_q} R$. 

We first show all coordinate hyperplanes are $(\hat{R},m)$-rich. For that, we pick hyperplane equations $h_i$ with coefficients in $\kF_q$ such that $\leftidx{^g}{\!h_i}{}=x_i$ for all $1\le i\le n$. As $R$ is $(n-1,m)$-Hom-Furstenberg we have $\text{dim}_{\kF_q} R\sqcap V_h \ge m$ where $V_{h_i}$ is the hyperplane defined by $h_i$. Noting the fact that $\text{in}(\leftidx{^g}{\!I}{})+\langle x_i\rangle \subseteq \text{in}(\leftidx{^g}{\!I}{}+\langle x_i\rangle)$ we have,
\begin{align*}
    \text{dim}_{\kF_q} \hat{R} \sqcap V_{x_i} =\text{dim}_{\kF_q} \frac{\kF_q[x_1,\hdots,x_n]}{(\text{in}(\leftidx{^g}{\!I}{})+\langle x_1\rangle )} &\ge \text{dim}_{\kF_q} \frac{\kF_q[x_1,\hdots,x_n]}{(\text{in}(\leftidx{^g}{\!I}{}+\leftidx{^g}{\!h_i}{}))}\\
    &=\text{dim}_{\kF_q} \frac{\kF_q[x_1,\hdots,x_n]}{I+\langle h_i\rangle } =\text{dim}_{\kF_q} R\cap V_{h_i} \ge m.
\end{align*}

For any hyperplane $V_f\in \kF_q^n$ with equation $f=f_1x_1+\hdots+f_nx_n$, it is easy to see that we can find an element $g_f\in \B(n,\kF_q)$ such that $\leftidx{^{g_f}}{\!f}{}$ is some coordinate hyperplane. 
Using Theorem \ref{genericInitialThm} we know $K$ is stable under the action of $g_f$. Using the fact that all coordinate hyperplanes are $(\hat{R},m)$-rich and $\leftidx{^{g_f}}{\!f}{}$ is a coordinate hyperplane we have,
$$\text{dim}_{\kF_q} \hat{R}\cap V_f=\text{dim}_{\kF_q} \frac{\kF_q[x_1,\hdots,x_n]}{K+\langle f\rangle}=\text{dim}_{\kF_q} \frac{\kF_q[x_1,\hdots,x_n]}{\leftidx{^{g_f}}{\!K}{}+\langle \leftidx{^{g_f}}{\!f}{}\rangle}=\text{dim}_{\kF_q} \frac{\kF_q[x_1,\hdots,x_n]}{K+\langle x_i\rangle}\ge m.$$
\end{proof}

For the third step, we consider subsets of the lattice $\Z_{\ge 0}^n$ with a simple geometric property. We let $\ce_i$ be the $i$th standard basis vector in an $n$ dimensional vector space. That is, $\ce_i$ is a vector of length $n$, with $1$ at position $i$ and $0$ everywhere else.

\begin{define}[Borel Exchange Property]
A subset $\Lambda$ of the set of non-negative lattice points $\Z_{\ge 0}^n$ has the {\em Borel Exchange Property (BEP)} if, for all $1 \le j<i\le n$ and  point $\lambda=(\lambda_1,\hdots,\lambda_n)\in \Lambda$ with $\lambda_i=0$, all lattice points of the form $\lambda+l(\ce_i-\ce_j)$ with $0\le l\le \lambda_j$ are in $\Lambda$. In other words, the intersection of $\Z_{\ge 0}^n$ with the ray starting from $\lambda$ in the direction $\ce_i-\ce_j$ is in $\Lambda$.
\end{define}

Lattices with BEP arise naturally from Borel stable monomial ideals. 

\begin{lem}\label{monoLattice}
Given a finite dimensional $\F$-algebra $R=\F[x_1,\hdots,x_n]/K$ where $K$ is a Borel stable monomial ideal, the set of vectors $\lambda=(\lambda_1,\hdots,\lambda_n)$, such that $x^\lambda=x_1^{\lambda_1}\hdots x_n^{\lambda_n}\in \text{Std}(R)$, forms a subset $\Lambda$ of $\Z_{\ge 0}^n$ which has the Borel Exchange Property.   
\end{lem}
\begin{proof}
By definition we have $\lambda=(\lambda_1,\hdots,\lambda_n)\in \Lambda$ if and only if $x^\lambda=x_1^{\lambda_1}\hdots x_n^{\lambda_n}\not\in K$. 
Take some $\lambda\in \Lambda$ such that $\lambda_i=0$. Assume for contradiction, there exists some $j<i$ and $l\le \lambda_j$ such that $\lambda+l(\ce_i-\ce_j)\not\in \Lambda$. This means $x^{\lambda+l(\ce_i-\ce_j)}=x_1^{\lambda_1}\hdots  x_j^{\lambda_j-l} \hdots x_i^{\lambda_i+l}\hdots x_n^{\lambda_n}\in K$. We can find an elementary upper triangular matrix $b$ in the Borel subgroup $\B(n,\F)$ such that $\leftidx{^b}{x}{}=x+x_j\ce_i$. As $x^{\lambda+l(\ce_i-\ce_j)}\in K$, $K$ being Borel stable implies $(\leftidx{^b}{x}{})^{\lambda+l(\ce_i-\ce_j)}=x_1^{\lambda_1}\hdots x_j^{\lambda_j-l} \hdots (x_i+x_j)^{\lambda_i+l}\hdots x_n^{\lambda_n}\in K$. Using the binomial expansion and the fact that $K$ is a monomial ideal we have $x^\lambda\in K$. But this implies $\lambda\not\in\Lambda$ leading to a contradiction.
\end{proof}

The next bound for lattices with BEP contains the main combinatorial argument.

\begin{lem}[Lattice bound]\label{latticeBound}
Given a set $\Lambda\subseteq \Z_{\ge 0}^n$ with the Borel Exchange Property, let $\Lambda_n$ be the subset of $\Lambda$ lying on the plane $\lambda_n=0$. If $|\Lambda_n|\ge m$ then 
$$|\Lambda|\ge 4^{-1} m^{n/(n-1)}.$$
\end{lem}

Before proving this Lemma, we first show how it implies Theorem~\ref{thm:furstPlaneZero}.

\begin{proof}[Proof of Theorem \ref{thm:furstPlaneZero}]
We start with an algebra $R=\F_q[x_1,\hdots,x_n]/I$ with $J_{l}(R)=\langle 0\rangle$. We extend to the algebraic closure using Lemma \ref{extendingLemma} to obtain a $(n-1,l)$-Hom-Furstenberg Algebra $\kR=\kF_q[x_1,\hdots,x_n]/I$ of dimension $\text{dim}_{\F_q} R$. Next, we use Lemma \ref{degenerationLemma} to obtain a $(n-1,l)$-Hom-Furstenberg Algebra $\hat{R}=\kF_q[x_1,\hdots,x_n]/K$ of dimension $\text{dim}_{\F_q} R$ where $K$ is a Borel stable monomial ideal. 

$\text{Std}(\hat{R})$ forms a basis of $\hat{R}$ and by Lemma \ref{monoLattice} this produces a lattice $\Lambda$ of size $\text{dim}_{\F_q} R$ with the Borel Exchange Property. The basis of $\hat{R}/\langle x_n\rangle$ is precisely the subset of monomials in $\text{Std}(\hat{R})$ which are not divisible by $x_n$.  Within $\Lambda$ they are precisely the subset $\Lambda_n$ of points $\lambda$ in $\Lambda$ lying on the plane $\lambda_n=0$. As $\hat{R}$ is $(n-1,l)$-Hom-Furstenberg, it has intersection of dimension at least $l$ with $x_n=0$. This implies $|\Lambda_n|\ge l$. Finally, Lemma \ref{latticeBound} gives us the required bound.
\end{proof}

\subsection{Proof of Lemma \ref{latticeBound}}

The proof involves using the Borel Exchange Property to find a subset of points in $\Lambda$ by applying the exchange property on points in $\Lambda_n$. We will show each point in $\Lambda_n$ will produce as many points as its weight and that any point in $\Lambda$ is generated by at most $n-1$ points in $\Lambda_n$. This lets us derive a lower bound for $\Lambda$.

We have $|\Lambda_n|\ge m \ge 1$. For any non-negative real $r$ and integer $a$ let 
\begin{align*}
{r \choose a}=\frac{r(r-1)(r-2)\hdots(r-a+1)}{a!}.
\end{align*}
 We can find a $d\ge 0$ such that 
\begin{align}
    m={n-1+d \choose n-1}.\label{eq:mFormShort}
\end{align}
Let $d'=\lfloor d \rfloor$. Then we can also write $m$ as
\begin{align}
    m={n-1+d' \choose n-1}+\beta {n-1+d' \choose n-2},\label{eq:mFormExpand}
\end{align}
where $\beta\in [0,1]$.

We will split the proof in 3 parts. First, for each point $\lambda \in \Lambda_n$ we will define a subset $\PATH(\lambda)\subseteq \Lambda$ associated with it and prove some simple properties of these sets. Next, we will lower bound the size of the union of these subsets to lower bound $|\Lambda|$. Finally, we will analyse the expression for $D(n,m)$ in our lower bound.

{\bf a) Definition of $\PATH$:} For any point $\lambda=(\lambda_1,\hdots,\lambda_{n-1},0)\in \Lambda_n$ we will produce points in $\Lambda$ and collect them in a set called $\PATH(\lambda)$. We will construct $\PATH(\lambda)$ in $n-1$ stages. In stage $1$, we start at $\lambda=P_1(\lambda)$ and we move along the direction $\ce_n-\ce_{n-1}$ until we hit the hyperplane $\lambda_{n-1}=0$ producing the lattice points $(\lambda_1,\lambda_2,\hdots,\lambda_{n-1}-l,l)\in \Lambda$ for all $0< l\le \lambda_{n-1}$, as $\Lambda$ is Borel fixed. We do not include $l=0$ to avoid repetition if $\lambda_{n-1}=0$. We collect these points in the set $\PATH(\lambda,1)$. We note $|\PATH(\lambda,1)|=\lambda_{n-1}$. At the end of stage $1$ we are at the point $P_2(\lambda)=(\lambda_1,\hdots,\lambda_{n-2},0,\lambda_{n-1})\in \Lambda$. 

In general, at stage $i$ we start at the point $P_i(\lambda)=(\lambda_1,\hdots,\lambda_{n-i},0,\lambda_{n-i+1},\hdots,\lambda_{n-1})$. We then move along the direction $\ce_{n-i+1}-\ce_{n-i}$ until we hit the hyperplane $\lambda_{n-i}=0$ producing lattice points $(\lambda_1,\hdots,\lambda_{n-i}-l,l,\lambda_{n-i+1},\hdots,\lambda_{n-1})\in \Lambda$ for $0<l\le \lambda_{n-i}$. We collect these points in $\PATH(\lambda,i)$ which has size $\lambda_{n-i}$. At the end of stage $i$ we are at the point $P_{i+1}(\lambda)=(\lambda_1,\hdots,\lambda_{n-i-1},0,\lambda_{n-i},\hdots,\lambda_{n-1})$. There are $n-1$ stages. We set $\PATH(\lambda)=\bigcup_{i=1}^{n-1}\PATH(\lambda,i)$. We note, 
\begin{align}
|\PATH(\lambda)|=\sum_{i=1}^{n-1} |\PATH(\lambda,i)|=\sum_{i=1}^{n-1} \lambda_i=\text{wt}(\lambda).\label{eq:wtPathEq}
\end{align}
For distinct points $\lambda$ and $\mu$ in $\Lambda_n$, $\PATH(\lambda)$ and $\PATH(\mu)$ may not be disjoint. For example, consider the points $\lambda=(2,0,2,0)$ and $\mu=(0,2,2,0)$. It is easy to see $\PATH(\lambda)$ and $\PATH(\mu)$ both contain $(0,2,0,2)$. We are going to show a single point will appear in $\PATH(\lambda)$ for at most $n-1$ points $\lambda$ in $\Lambda_n$ using the following claim.

\begin{claim}
For $\alpha\in \bigcup_{\lambda \in \Lambda_n}\PATH(\lambda)$ and each $1\le i\le n-1$, there is at most one point $\lambda \in \Lambda_n$ such that $\alpha \in \PATH(\lambda,i)$.
\end{claim}
\begin{proof}
Say there exist two such points $\lambda=(\lambda_1,\hdots,\lambda_{n-1},0)$ and $\mu=(\mu_1,\hdots,\mu_{n-1},0)$. Let $\alpha=(\alpha_1,\hdots,\alpha_n)$. As $\lambda$ and $\mu$ produced $\alpha$ in stage $i$, there exist $l_1$ and $l_2$ such that
\begin{align*}
    \alpha &= (\alpha_1,\hdots,\alpha_{n-i},\alpha_{n-i+1},\alpha_{n-i+2}\hdots,\alpha_{n})\\
    &= (\lambda_1,\hdots,\lambda_{n-i}-l_1,l_1,\lambda_{n-i+1},\hdots,\lambda_{n-1}) \\
    &= (\mu_1,\hdots,\mu_{n-i}-l_2,l_2,\mu_{n-i+1},\hdots,\mu_{n-1}).
\end{align*}
This implies $\lambda_j=\mu_j=\alpha_j$ for $1\le j\le n-i-1$, $\lambda_{j}=\mu_j=\alpha_{j+1}$ for $n-i+1\le j\le n-1$, and $\lambda_{n-i}=\mu_{n-i}=\alpha_{n-i}+\alpha_{n-i+1}$. This proves $\lambda=\mu$.
\end{proof}

This claim then implies there are at most $n-1$ points $\lambda$ in $\Lambda_n$ such that $\alpha \in \PATH(\lambda)$. This gives us the following lower bound,
\begin{align}
 \Big|\bigcup_{\lambda\in \Lambda_n} \PATH(\lambda)\Big|\ge \frac{1}{n-1}\Big(\sum\limits_{\lambda\in \Lambda_n} |\PATH(\lambda)|\Big)    .\label{eq:unionLatticeBound}
\end{align}

{\bf b) Lower bound on union of $\PATH(\lambda)$:} We have at least $m$ points in $\Lambda_n$. We want to lower bound the number of points in the set $\Lambda$ by lower bounding the size of its subset $\bigcup_{\lambda\in \Lambda_n} \PATH(\lambda)$. Equation \ref{eq:unionLatticeBound} shows that it suffices to lower bound $\sum_{\lambda\in \Lambda_n} |\PATH(\lambda)|$. We find a lower bound by computing its minimum possible value. Equation \ref{eq:wtPathEq} shows that $|\PATH(\lambda)|=\text{wt}(\lambda)$. We know there are ${n-2+k \choose n-2}$ many points with weight $k$. Equation \ref{eq:mFormExpand} shows $m$ is of the form,
\begin{align}
    m={n-1+d' \choose n-1} + \beta {n-1+d' \choose n-2} =\sum\limits_{k=0}^{d'} {n-2+k \choose n-2} + \beta {n-1+d' \choose n-2}.\label{eq:mFormExpand2}
\end{align}
To minimize $\sum_{\lambda\in \Lambda_n} |\PATH(\lambda)|=\sum_{\lambda\in \Lambda_n} \text{wt}(\lambda)$, $\Lambda_n$ needs to contain points with as small weight as possible. This implies that for the minimizer $\Lambda_n$ should contain all points of weight at most $d'$ and only a $\beta$ fraction of points with weight $d'+1$. This gives us the following bound,
\begin{align*}
|\Lambda|&\ge \Big|\bigcup_{\lambda\in \Lambda_n} \PATH(\lambda)\Big|\\
&\ge \frac{1}{n-1}\left(\sum\limits_{\lambda\in \Lambda_n} |\PATH(\lambda)|\right)\\
&\ge \frac{1}{n-1}\left(\sum\limits_{k=0}^{d'} k{n-2+k \choose n-2} + \beta (d'+1){n-1+d' \choose n-2}\right)\\
  &= \frac{1}{n-1}\left(\sum\limits_{k=1}^{d'} (n-1){n-1+k-1 \choose n-1} + \beta (n-1){n-1+d' \choose n-1}\right)\\
  &={n-1+d' \choose n} + \beta {n-1+d' \choose n-1}=\frac{d'+\beta n}{n}{n-1+d' \choose n-1}.\numberthis \label{eq:LambdaBoundMid}
\end{align*}
Rearranging Equation \ref{eq:mFormExpand} gives us,
$${n-1+d' \choose n-1}=m\Big(\frac{d'+1}{d'+1+\beta (n-1)}\Big).$$
This allows us to write Equation \ref{eq:LambdaBoundMid} as,
\begin{align*}
    |\Lambda|&\ge m\frac{d'+\beta n}{n}\frac{d'+1}{d'+1+\beta (n-1)}\\
    &= m^{n/(n-1)}{n-1+d' \choose n-1}^{-1/(n-1)}\frac{d'+\beta n}{n}\Big(\frac{d'+1}{d'+1+\beta (n-1)}\Big)^{n/(n-1)}.
\end{align*}
This equation has the form we need to prove the theorem. To complete the proof, we examine,
\begin{align}
D(n,m) &= {n-1+d' \choose n-1}^{-1/(n-1)}\frac{d'+\beta n}{n}\Big(\frac{d'+1}{d'+1+\beta (n-1)}\Big)^{n/(n-1)}.\label{eq:DnmEquation}
\end{align}


{\bf c) Analysing $D(n,m)$: } First, we want to prove $D(n,m)\ge 1/4$. We note it suffices to do this while assuming $m\ge 4^{n-1}$ because when $m<4^{n-1}$ setting $D(n,m)=1/4$ we have $|\Lambda|\ge m \ge m^{n/(n-1)}/4$. We also note $m\ge 4^{n-1}$ forces $d>n$. This is because if $d\le n$ then $m\le {2n-1 \choose n-1}\le 4^{n-1}$. Using Equation \ref{eq:DnmEquation} we have,
\begin{align*}
    D(n,m)&={n-1+d' \choose n-1}^{-1/(n-1)}\frac{d'+\beta n}{n}\Big(\frac{d'+1}{d'+1+\beta (n-1)}\Big)^{n/(n-1)}\\
    &={n-1+d' \choose n-1}^{-1/(n-1)}\frac{(d'+1)^{n/(n-1)}}{n}\frac{d'+\beta n}{(d'+1+\beta (n-1))^{n/(n-1)}}\numberthis \label{eq:DnmAnalyseBet} \\
    &=\frac{(n-1)!^{1/(n-1)}}{n} \frac{d'+1}{\prod_{k=1}^{n-1} (d'+k)^{1/(n-1)}} \Big(\frac{d'+1}{d'+1+\beta (n-1)}\Big)^{1/(n-1)}\frac{d'+\beta n}{d'+1+\beta (n-1)}.\numberthis \label{eq:DnmAnalyse} 
\end{align*}
To help analyse $D(n,m)$, we want to show $D(n,m)$ is increasing in $d'$ and $\beta$. 


\begin{claim}
The expression of $D(n,m)$ is increasing in $d'$.
\end{claim}
\begin{proof}
In Equation \ref{eq:DnmAnalyse}, as $\beta\le 1$ we have $(d'+\beta n)/(d'+1+\beta (n-1))=1-(1-\beta)/(d'+1+\beta (n-1))$ is increasing in $d'$. $(d'+1)/(d'+1+\beta (n-1))=1-\beta(n-1)/(d'+1+\beta (n-1))$ is also increasing in $d'$. $(d'+1)/(\prod_{k=1}^{n-1} (d'+k)^{1/(n-1)})=\prod_{k=1}^{n-1} \Big(1-(k-1)/(d'+k)\Big)^{1/(n-1)}$ is also increasing in $d'$. This shows $D(n,m)$ is increasing in $d'$.
\end{proof}

\begin{claim}
The expression of $D(n,m)$ is increasing in $\beta$ for $\beta \le 1$.
\end{claim}
\begin{proof}
In Equation \ref{eq:DnmAnalyseBet}, consider the term $(d'+\beta n)/(d'+1+\beta (n-1))^{n/(n-1)}$. Taking logarithm we get $\log(d'+\beta n)-n\log(d'+1+\beta (n-1))/(n-1)$ which is a function in $\beta$. Taking derivative with respect to $\beta$ gives us $n/(d'+\beta n)-n/(d'+1+\beta(n-1))=n(1-\beta)/((d'+\beta n)(d'+1+\beta(n-1)))\ge 0$. This shows $D(n,m)$ is increasing in $\beta$ for $\beta\le 1$.
\end{proof}

We have the condition $d>n$ and $\beta\in [0,1]$. To find a lower bound for $D(n,m)$ we set $d'=n$ and $\beta=0$ in Equation \ref{eq:DnmEquation} to get the bound,
$$D(n,m)\ge {2n-1 \choose n-1}^{-1/(n-1)}\ge 1/4.$$

\qed

In the argument above it is easy to show for a fixed $n$ that $\lim_{m\rightarrow \infty} D(n,m)= (n-1)!^{1/(n-1)}/n$. We now note Lemma~\ref{latticeBound} is asymptotically optimal upto constants, as the algebra $\F[x_1,\hdots,x_n]/\langle x_1,\hdots,x_n\rangle^m$ produces a lattice with the Borel Fixed Property and from the argument in Example \ref{constantBound} we have $D(n,m)\le (n-1)!^{1/(n-1)}/n$. 
\section{Theorem \ref{thm:furstPlaneNon} using Method of Multiplicities}\label{sec:MultCase}
To prove Theorem \ref{thm:furstPlaneNon} we prove for any $(n-1,m)$-Hom-Furstenberg Algebra $R$, every polynomial in $J_k(R)$ vanishes on every point in $\V_{\F_q}(J_m(R))$ with high multiplicity. Finally, we use the Schwartz-Zippel bound to obtain the desired bound.

We recall that if $f\in J_l(R)$ and a hyperplane with equation $h$ is $(R,l)$-rich then $f(h)=0$. In other words, $\text{mult}(f,h)\ge 1$. The next lemma proves that if $h$ is $(R,m)$-rich for $m>l$ then $f$ vanishes on $h$ with higher multiplicity.
\begin{lem}\label{MultiplicityMainLemma}
Given a finite dimensional $\F$-algebra $R=\F[x_1,\hdots,x_n]/I$, an $(R,m)$-rich hyperplane $V_h$ with equation $h:h_1x_1+\hdots+h_nx_n=0$, and a polynomial $f\in J_l(R)\subseteq \F[h_1,\hdots,h_n]$ for $0\le l\le m$ we have,
$$\text{mult}(f,h)\ge m-l+1.$$
\end{lem}
\begin{proof}
Recall, $T_h:R\rightarrow R$ defined in the proof of Theorem \ref{Jmdef} is the multiplication map, mapping $f\in R$ to $(h_1x_1+\hdots+h_nx_n)f\in R$. $J_l(R)$ is generated by $L=(\text{dim}_\F R)-l+1$ sized minors of the matrix of $T_h$ in any basis. Fix any basis and let the matrix in this case be $U(h)$ with entries which are polynomials in $h_1,\hdots,h_n$. Consider a formal $L\times L$ matrix $Y$ with formal variable entries $y_{ij}$ for $1\le i,j\le L$. Let $\text{Det}_L$ be a polynomial over the variables $y_{ij}$ obtained by taking the determinant of $Y$. 

Any $(\text{dim}_\F R)-l+1$ minor is the composition of $\text{Det}_L$ and a $L\times L$ matrix of polynomials $U_L(h)$ corresponding to some $L\times L$ sub-matrix of $U(h)$. Given a $(R,m)$-rich hyperplane $V_g$ with equation $g(x)=g_1x_1+\hdots+g_nx_n$, all minors of the matrix of $U(g)$ of size $M=(\text{dim}_\F R)-m+1$ vanish. As $U_L(g)$ is a submatrix, all its minors of size $M<L$ also vanish. It is not hard to check that the weight $t$ Hasse derivatives of the determinant polynomial $\text{Det}_N$ of a matrix are generated by minors of the matrix of size $N-t$. This means all the Hasse derivatives of $\text{Det}_N$ of weight strictly less than $m-l$ vanish on $U_L(g)$. This means $\text{mult}(\text{Det}_N, U_L(g)) \ge m-l+1$. Using Lemma \ref{lem:multComp} we have $\text{mult}(\text{Det}_N\circ U_L,g)\ge \text{mult}(\text{Det}_N, U_L(g)) \ge m-l+1$. As the polynomials $\text{Det}_N\circ U_L$ for different submatrices $U_L$ of $U$ generate $J_l(R)$ we are done.

\end{proof}

We finally prove Theorem \ref{thm:furstPlaneNon}.

\begin{proof}[Proof of Theorem \ref{thm:furstPlaneNon}]
 Given a $(n-1,m)$-Hom-Furstenberg algebra $R$ over the field $\F_q$, we know $J_{l}(R)$ contains a non-zero polynomial $f$ of degree $(\text{dim}_{\F_q} R)-l+1$. Lemma \ref{MultiplicityMainLemma} tells us $f$ vanishes on the equations of all $(R,m)$-rich hyperplanes with multiplicity $m-l+1$. We know all hyperplanes in $\F^n_q$ are $(R,m)$-rich. This means $f$ vanishes on $\F^n_q\setminus \{0\}$ with multiplicity at least $m-l+1$. As $f$ is a homogenous polynomial of degree $(\text{dim}_{\F_q} R)-l+1>m-l+1$ it is easy to check $f$ also vanishes on the origin with multiplicity at least $m-l+1$. Finally, using Lemma \ref{multSchwartz} we have
$$q^{n}(m-l+1) \le((\text{dim}_{\F_q} R)-l+1)q^{n-1}.$$ 
Rearranging, we get
$$\text{dim}_{\F_q} R \ge mq\left(1-\frac{l-1}{m}\right).$$
\end{proof}

\section{Hyper-Furstenberg Bound }\label{sec:higherdegree}

In this section we will prove Theorem \ref{thm:hyperBound} by reducing it to \ref{thm:furstPlaneBound}. The proof involves two steps first we reduce to some problem over homogenous Hyper-Furstenberg Algebras. Finally, we will show these algebras are also Furstenberg Algberas with appropriate parameters.

\begin{define}[Homogenous Hyper-Furstenberg Algebra] A finite dimensional $\F$-algebra $R=\F[x_1,\hdots,x_n]/I$ with $I$ homogenous is said to be $(k,m,d)$-Hyper-Hom-Furstenberg, if for all  $k$ dimensional subspaces $U\subseteq \F_q^n$ there exists linearly independent hyperplanes $h_1,\hdots,h_{n-k}$ whose intersection equals $U$ and $d_1,\hdots,d_{n-k}\in \N,\prod_{i=1}^{n-k}d_i\le d$ such that, $\langle h_1^{d_1},\hdots,h_{n-k}^{d_{n-k}}\rangle$ is $(R,m)$-rich. 
\end{define}

\begin{lem}[Reduction to Hyper Hom Furstenberg] Given a finite dimensional $(k,m,d)$-Hyper-Furstenberg Algebra $R=\F[x_1,\hdots,x_n]/I$, $\text{hd}(R)=\F[x_1,\hdots,x_n]/\text{hd}(I)$ is a $(k,m,d)$-Hyper-Hom-Furstenberg Algebra such that {\em $\text{dim}_\F \text{hd}(R)=\text{dim}_\F R$.}
\end{lem}\label{lem:hyperHomFurst}
\begin{proof}
Given a $k$ dimensional subspace $U$, let $h_1,\hdots,h_{n-k}$ be linearly independent hyperplanes whose intersection is $U$ and $g_i, 1\le i\le n-k$ be polynomials of degree $d_i-1$ such that $\langle h_1^{d_1}+g_1,\hdots,h_{n-k}^{d_{n-k}}+g_{n-k}\rangle$ is $(R,m)$-rich and $\prod_{i=1}^{n-k}d_i \le d$. We note that, $\text{hd}(\langle h_1^{d_1}+g_1,\hdots,h_{n-k}^{d_{n-k}}+g_{n-k}\rangle)=\langle h_1^{d_1},\hdots,h_{n-k}^{d_{n-k}}\rangle$.
The second claim of Lemma~\ref{dilationLem} now implies
\begin{align*}
m\le \text{dim}_\F \left(R/\langle h_1^{d_1}+g_1,\hdots,h_{n-k}^{d_{n-k}}+g_{n-k}\rangle\right)& \le \text{dim}_\F \left(\text{hd}(R)/\text{hd}(\langle h_1^{d_1}+g_1,\hdots,h_{n-k}^{d_{n-k}}+g_{n-k}\rangle)\right)\\
&=\text{dim}_\F \left(\text{hd}(R)/\langle h_1^{d_1},\hdots,h_{n-k}^{d_{n-k}}\rangle\right).
\end{align*}

This shows $\langle h_1^{d_1},\hdots,h_{n-k}^{d_{n-k}}\rangle$ is $(\text{hd}(R),m)$-rich.
\end{proof}

\begin{lem}[Reduction to Theorem~\ref{AlgebraBound}]\label{lem:hyperToPlane}
Any $(k,m,d)$-Hyper-Hom-Furstenberg Algebra $R=\F[x_1,\hdots,x_n]/I$ is a $(k,m/d)$-Hom-Furstenberg Algebra.
\end{lem}
\begin{proof}
Given any $U\subseteq \F_q^n$ dimensional subspaces there exists linearly independent hyperplanes $h_1,\hdots,h_{n-k}$ whose intersection equals $U$ and $d_1,\hdots,d_{n-k}\in \N,\prod_{i=1}^{n-k}d_i\le d$ such that $\langle h_1^{d_1},\hdots,d_{n-k}^{n-k}\rangle$ is $(R,m)$-rich. We can perform a base change over $\F$ such that $h_i=x_i$ for $1\le i\le n-k$. Take the basis of standard monomials $\text{Std}(R)$. 

$\langle x_1^{d_1},\hdots,x_{n-k}^{d_{n-k}}\rangle$ being $(R,m)$-rich implies that there are at least $m$ monomials in $\text{Std}(R)$ such that the degree of $x_i$ in each of these monomials is strictly less than $d_i$ for $1\le i\le n-k$. We note if $x_1^{\lambda_1}x_2^{\lambda_2}\hdots x_n^{\lambda_n}$ is a standard monomial then so is $x_{n-k+1}^{\lambda_{n-k+1}}\hdots x_n^{\lambda_n}$. This is because being a standard monomial means not lying in $\text{in}(I)$ and if a polynomial is not in an ideal than all its factors will also not be in it. This shows there are at least $m/d$ monomials in $\text{Std}(R)$ such that the degree of $x_i,1\le i\le n-k$ in each them is $0$. This shows $\langle x_1,\hdots,x_{n-k}\rangle$ is $(R,m/d)$-rich.

This shows that every $k$ dimensional subspace is $(R,m/d)$-rich.
\end{proof}

We also need a simple lemma connecting the value of $m$ and $d$.

\begin{lem}\label{lem-hypermbound}
For any $(k,m,d)$-Hyper-Furstenberg set over $\F_q^n$ we have that $m\le dq^k$.
\end{lem}
\begin{proof}
Consider a fixed $k$ dimensional sub-space $U$. If $S$ is a $(k,m,d)$-Hyper-Furstenberg set then there exist linearly independent hyperplanes $h_1,\hdots,h_{n-k}$ whose intersection is $U$ and $d_1,\hdots,d_{n-k}\in \N,\prod_{i=1}^{n-k}d_i\le d$ and polynomials $g_i,1\le i\le n-k$ of degree at most $d_i-1$ such that $\langle h_1^{d_1}+g_1,\hdots,h_{n-k}^{d_{n-k}}+g_{n-k}\rangle$ intersects with $S$ in at least $m$ points. After a bass change we can take $h_i=x_i$. It is clear that $m$ is upper bounded by the number of points in $\F_q^n$ which vanish on $\langle h_1^{d_1}+g_1,\hdots,h_{n-k}^{d_{n-k}}+g_{n-k}\rangle$. In other words,

$$m\le \text{dim}_{\F_q} \frac{\F_q[x_1,\hdots,x_n]}{\langle h_1^{d_1}+g_1,\hdots,h_{n-k}^{d_{n-k}}+g_{n-k},x_1^q-x_1,\hdots,x_n^q-x_n\rangle}\le dq^k.$$

The last inequality follows by noting that monomials of the form $x_1^{\lambda_1}\hdots x_{n}^{\lambda_n}$ where $0\le \lambda_i\le d_i-1,1\le i\le n-k$ and $0\le \lambda_j\le q-1,n-k+1\le j \le n$ spans the finite algebra 
$$\frac{\F_q[x_1,\hdots,x_n]}{\langle h_1^{d_1}+g_1,\hdots,h_{n-k}^{d_{n-k}}+g_{n-k},x_1^q-x_1,\hdots,x_n^q-x_n\rangle}.$$
\end{proof}

\begin{proof}[Proof of Theorem \ref{thm:hyperBound}]

Given a $(k,m,d)$-Hyper-Furstenberg Set $S\subseteq \F_q^n$, Corollary~\ref{cor:hyperFurstSet2alg} and Lemma~\ref{lem:hyperHomFurst} implies that $\text{hd}(\text{Alg}(S))$ is $(k,m,d)$-Hom-Furstenberg Algebra of dimension $|S|$. Lemma~\ref{lem:hyperToPlane} implies that $\text{hd}(\text{Alg}(S))$ is $(k,m/d)$-Hom-Furstenberg Algebra. By Lemma~\ref{lem-hypermbound}, we know $m\le q^{k}d$ which means we can now apply Theorem~\ref{AlgebraBound} to get,
$$|S|\ge C_{n,k}\frac{m^{n/k}}{d^{n/k}},$$
where $C_{n,k}=\Omega(1/16^{n\ln(n/k)})$.
\end{proof}

\section{Proof of Theorem \ref{th:combinatorial}}\label{sec:ProofCombiBound}

Before proving Theorem \ref{th:combinatorial}, we recall some elementary observations on the linear subspaces of $\F_q^n$.
The number of $k$-dimensional linear subspaces of $\mathbb{F}_q^n$ is
\[ \binom{n}{k}_q = \frac{(q^n-1)(q^n-q)\ldots(q^n-q^{k-1})}{(q^k-1)(q^k-q)\ldots(q^k-q^{k-1})}.\]
Indeed, the numerator counts the number of ordered $k$-tuples of  linearly independent vectors in $\F_q^n$, and the denominator counts the number of ordered $k$-tuples of linearly independent vectors in $\F_q^k$.
Similarly, the number of $k$-dimensional subspaces that contain a fixed $\ell$-dimensional subspace (for $\ell \leq k$) is $\binom{n-\ell}{k-\ell}_q$.
Note that
\begin{align*} \binom{n}{k}_q &= \frac{(q^n-1)(q^n-q)\ldots(q^n - q^{\ell-1})}{(q^k-1)(q^k-q)\ldots(q^k-q^{\ell-1})} \frac{q^\ell(q^{n-\ell}-1)q^\ell(q^{n-\ell}-q)\ldots q^\ell(q^{n-\ell} - q^{k-\ell-1})}{q^\ell(q^{k-\ell}-1)q^\ell(q^{k-\ell}-q)\ldots q^\ell(q^{k-\ell}-q^{k-\ell-1})}\\
&= \frac{(q^n-1)(q^n-q)\ldots(q^n - q^{\ell-1})}{(q^k-1)(q^k-q)\ldots(q^k-q^{\ell-1})} \binom{n-\ell}{k-\ell}_q.
\end{align*}

\begin{thm}
If $S \subset \F_q^n$ is a (k,m)-Furstenberg set in $\F_q^n$ and $\ell$ is an integer with $1 \leq \ell < \log_q(m)+1$, then
\[ |S|^{\ell +1} \geq q^{\ell(n-k)}m(m-1)(m-q)\ldots(m-q^{\ell-1}) .\]
\end{thm}
\begin{proof}
Let $E$ be the set of pairs $(H,P)$, where $H$ is an $(S,m)$-rich affine $k$-plane and $P\subset (S \cap H)$ is an ordered $(\ell+1)$-tuple of affinely independent points of $S$ contained in $H$.

Since $S$ is $(k,m)$-Furstenberg, there is an $(S,m)$-rich $k$-plane parallel to each of the $\binom{n}{k}_q$ subspaces of dimension $k$.
For a fixed $(S,m)$-rich $k$-plane $H$, there are at least $m(m-1)(m-q)\ldots(m-q^{\ell-1})$ ordered tuples of $(\ell+1)$ affinely independent points of $S$ contained in $H$.
Indeed, the affine span of any $t$ affinely independent points contains $q^{t-1}$ points of $\F_q^n$, hence there are at least $m-q^{t-1}$ choices for the next point once the first $t$ are chosen.
Combining these observations, we have
\[ |E| \geq \binom{n}{k}_q m(m-1)(m-q)\ldots(m-q^{\ell-1}).\]

On the other hand, each set of $\ell+1$ affinely independent points spans an $\ell$-plane, which is contained in $\binom{n-\ell}{k-\ell}_q$ affine $k$-planes.
Since the total number of ordered $(\ell + 1)$-tuples of points in $S$ is less than $|S|^{\ell+1}$, we have
\[ |E| \leq |S|^{\ell+1}\binom{n-\ell}{k-\ell}_q.\]

Combining the upper and lower bounds yields
\begin{align*} |S|^{\ell+1} &\geq \frac{(q^n-1)(q^n-q)\ldots(q^n - q^{\ell-1})}{(q^k-1)(q^k-q)\ldots(q^k-q^{\ell-1})} m(m-1)(m-q)\ldots(m-q^{\ell-1}) \\ &\geq q^{\ell(n-k)}m(m-1)(m-q)\ldots(m-q^{\ell-1}),
\end{align*}
as claimed.
\end{proof}

\bibliographystyle{alpha}


\begin{thebibliography}{DKSS13}

\bibitem[BC21]{BC21}
Boris Bukh, and Ting-Wei Chao.
\newblock Sharp density bounds on the finite field Kakeya problem.
\newblock {\em Discrete Analysis}, 26, 2021.

\bibitem[DDL21]{DDL-2}
Manik Dhar, Zeev Dvir, and Ben Lund.
\newblock Simple proofs for Furstenberg sets over finite fields.
\newblock {\em Discrete Analysis}, 22, 2021.

\bibitem[DKSS13]{dvir2013extensions}
Zeev Dvir, Swastik Kopparty, Shubhangi Saraf, and Madhu Sudan.
\newblock Extensions to the method of multiplicities, with applications to
  {K}akeya sets and mergers.
\newblock {\em SIAM Journal on Computing}, 42(6):2305--2328, 2013.

\bibitem[Dvi09]{dvir2009size}
Zeev Dvir.
\newblock On the size of {K}akeya sets in finite fields.
\newblock {\em Journal of the American Mathematical Society}, 22(4):1093--1097,
  2009.

\bibitem[Dvi12]{dvir-survey}
Zeev Dvir.
\newblock Incidence theorems and their applications.
\newblock {\em Foundations and Trends® in Theoretical Computer Science},
  6(4):257--393, 2012.

\bibitem[EE16]{ellenberg2016furstenberg}
Jordan Ellenberg and Daniel Erman.
\newblock Furstenberg sets and {F}urstenberg schemes over finite fields.
\newblock {\em Algebra \& Number Theory}, 10(7):1415--1436, 2016.

\bibitem[Eis95]{eisenbud1995commutative}
David Eisenbud.
\newblock {\em Commutative Algebra: With a View Toward Algebraic Geometry}.
\newblock Graduate Texts in Mathematics. Springer, 1995.

\bibitem[EOT10]{ellenberg_oberlin_tao_2010}
Jordan~S. Ellenberg, Richard Oberlin, and Terence Tao.
\newblock The {K}akeya set and maximal conjectures for algebraic varieties over
  finite fields.
\newblock {\em Mathematika}, 56(1):1–25, 2010.

\bibitem[HH02]{Hochster2002}
Melvin Hochster and Craig Huneke.
\newblock Comparison of symbolic and ordinary powers of ideals.
\newblock {\em Inventiones mathematicae}, 147(2):349--369, Feb 2002.

\bibitem[KLSS11]{kopparty2011kakeya}
Swastik Kopparty, Vsevolod~F Lev, Shubhangi Saraf, and Madhu Sudan.
\newblock Kakeya-type sets in finite vector spaces.
\newblock {\em Journal of Algebraic Combinatorics}, 34(3):337--355, 2011.

\bibitem[Sch79]{schwartz1979probabilistic}
Jacob~T Schwartz.
\newblock Probabilistic algorithms for verification of polynomial identities.
\newblock In {\em International Symposium on Symbolic and Algebraic
  Manipulation}, pages 200--215. Springer, 1979.

\bibitem[SS08]{saraf2008}
Shubhangi Saraf and Madhu Sudan.
\newblock An improved lower bound on the size of {K}akeya sets over finite
  fields.
\newblock {\em Anal. PDE}, 1(3):375--379, 2008.

\bibitem[Wol99]{wolf1999}
Thomas Wolff.
\newblock Recent work connected with the {K}akeya problem.
\newblock {\em Prospects in mathematics (Princeton,NJ, 1996)}, pages 29--162,
  1999.

\bibitem[Zip79]{ZippelPaper}
Richard Zippel.
\newblock Probabilistic algorithms for sparse polynomials.
\newblock In Edward~W. Ng, editor, {\em Symbolic and Algebraic Computation},
  pages 216--226, Berlin, Heidelberg, 1979. Springer Berlin Heidelberg.

\end{thebibliography}

\appendix
\section{Appendix: Proof of Theorem \ref{genericInitialThm}}\label{apx:GIN}

\begin{proof}
Let $P=\F[x_1,\hdots,x_n]$. Let $I_d$ be the degree-$d$ part of the homogenous ideal $I$. Let $b$ be a general element in the Borel group described by the indeterminates $b_{ij},1\le i\le j\le n$. The action of $b$ on $I_d$ is a linear map from $I_d$ to the space $P_d$ of degree $d$ homogenous polynomials in $P$. We take degree $d$ monomials in $P$ as a basis of $P_d$ and order them in decreasing order according to the graded lexicographic order. We also take a basis $f_1,\hdots,f_t$ of $I_d$. The action of $b$ on $I_d$ can be written as a matrix $H_d(b)$ in these basis elements. Given a monomial $m$, let the entry at the $f_t$ column and row $m$ of $H_d(b)$ be $h_{m,f_i}(b)$. Observe $h_{m,f_i}(b)$ is a polynomial in the indeterminates $b_{ij}$. $h_{m,f_i}(b)$ is the coefficient of $m$ in the polynomial $\leftidx{^b}{\!f_i}{}$.
Here is a helpful diagram to keep in mind.
\begin{align*}
    H_d(b)= \begin{blockarray}{cccc}
     f_1 & f_2 & \hdots & f_t \\
     \begin{block}{(cccc)}
    h_{x_1^d,f_1}(b) & h_{x_1^d,f_2}(b) & \hdots & h_{x_1^d,f_t}(b)\\
    h_{x_1^{d-1}x_2,f_1}(b)   & h_{x_1^{d-1}x_2,f_2}(b) & \hdots & h_{x_1^{d-1}x_2,f_t}(b)\\
    \vdots & \vdots & \vdots & \vdots\\
    \end{block}
    \end{blockarray}
    \qquad\text{\parbox{5cm}{Rows are indexed by monomials of degree $d$ in decreasing grlex order}}
\end{align*}

Consider the set $M(H_d(b))$ of $t\times t$ submatrices of $H_d(b)$. This set can be indexed by an ordered $t$-tuple of degree $d$ monomials corresponding to the rows chosen to form the $t\times t$ submatrices. The grlex monomial order can be used to define an order over ordered $t$-tuples of monomials $a_1>\hdots>a_t$ by ordering them lexicographically. Let $M'(H_d(b))$ be the set of submatrices in $M(H_d(b))$ with non-zero determinant, as a polynomial in $b$. This set will be non-empty as the action of $b$ on $K_d$ is invertible. Each submatrix in $M'(H_d(b))$ corresponds to an ordered $t$-tuple of monomials. Let $m_1>\hdots>m_t$ be the largest among them. If we perform row or column operations over $H_d(b)$ in the field of rational functions over $b$, the tuples determining elements in $M'$ will not change. 

Let the vector space spanned by the monomials $m_1,\hdots,m_t$ be $K_d$. Let $q_d$ be the non-zero polynomial obtained by taking the determinant of the submatrix of $H_d(b)$ corresponding to the rows $m_1>\hdots>m_t$. We construct $K_d$ and $q_d$ for every $d\ge 0$. For a general $g\in \B(n,\F)$ we can also consider the set $M'(H_d(g))$ of $t\times t$ submatrices with non-zero determinants in $H_d(g)$. Each of the elements corresponds to a unique ordered $t$-tuple of monomials. Again, this set is non-empty as the action of $g$ on $I_d$ is invertible. Let $k_1(g)>\hdots>k_t(g)$ be the largest such tuple in $M'(H_d(g))$. We make the following claim.
\begin{claim}
$k_1(g),\hdots,k_t(g)$ form a basis of $\text{in}(\leftidx{^g}{\!I_d}{})$.
\end{claim}
\begin{proof}
 First, note the column space of $H_d(g)$ is by construction $\leftidx{^g}{\!I_d}{}$. Then as the submatrix corresponding to the rows has non-zero determinant we can perform a series of column operations, or equivalently left multiplying by a $t\times t$ invertible matrix to change the submatrix corresponding to the rows $k_1(g)>\hdots>k_t(g)$ into identity without changing the column space. Let this new matrix be $H'_d(g)$. As $k_1(g)>\hdots>k_t(g)$ was the largest tuple in $M'(H_d(g))$ we claim that in column $i$ of $H'_d(g)$ there is no non-zero element over row $k_i(g)$. By construction there can't be one in rows $k_1(g),\hdots,k_{i-1}(g)$. There can't be a non-zero entry in other rows because if there were then replacing $k_i(g)$ by that row in $k_1>\hdots>k_t(g)$ (and possibly re-ordering) we will get $t\times t$ submatrix with non-zero determinant indexed by a larger tuple. This would contradict the maximality of $k_1(g)>\hdots>k_t(g)$.
\end{proof}
 
As $\F$ is infinite we can find an element $g\in \B(n,\F)$ such that $q_d(g)\ne 0$. Consider the set $M'(H_d(g))$ of $t\times t$ submatrices with non-zero determinants in $H_d(g)$. Each of the elements corresponds to a unique ordered $t$-tuple of monomials. Again, this set is non-empty as the action of $g$ on $I_d$ is invertible. The largest ordered tuple will be $m_1>\hdots>m_t$. This is the case because if there were a larger tuple then that would mean the sub-matrix corresponding to that in $H_d(b)$ will have a non-zero determinant. This would contradict the maximality of $m_1>\hdots>m_t$. This shows that if $q_d(g)\ne 0$, then $\text{in}(\leftidx{^g}{\!I_d}{})$ is spanned by $m_1,\hdots,m_t$ which means it equals $K_d$.

We claim $K=\bigoplus_{d\ge 0} K_d$ is an ideal. It suffices to show that $P_1K_d\subset K_{d+1}$. We can find an element $g\in \B(n,\F)$ such that $q_d(g)q_{d+1}(g)\ne 0$. We then have $\text{in}(\leftidx{^g}{\!I_d}{})$ spans $K_d$ and $\text{in}(\leftidx{^g}{\!I_{d+1}}{})$ spans $K_{d+1}$. As $\text{in}(\leftidx{^g}{\!I}{})$ is an ideal the claim follows. As each $K_d$ is spanned by monomials we have that $K$ is a monomial ideal.

We finally finish proving the first claim. We know that ideals over $P=\F[x_1,\hdots,x_n]$ are finitely generated. Say the generators of $K$ and $I$ are of degree at most $l$. We set $q=\prod_{d=0}^l q_d$. As the generators of $I$ and $K$ are of degree at most $l$, we have for any $g\in \B(n,\F),q(g)\ne 0$ that $\text{in}(\leftidx{^g}{\!I}{})=K$ proving the first claim.

To prove the second claim, we want to show for all $g\in \B(n,\F)$, $\leftidx{^g}{\!K}{}=K$. As a simplification we can take $I=K$ as the generic initial ideal of $K$ is $K$ itself. As before consider the matrix $H_d(b)$ for an indeterminate $b$. For an element $g\in \B(n,\F)$, consider again the set $M'(H_d(g))$. Pick the largest ordered tuple $k_1(g)>k_2(g)\hdots>k_t(g)$ corresponding to a sub-matrix in $M'(H_d(g))$. We claim this tuple must be smaller than $m_1>\hdots>m_t$. If it were larger then that would imply that the determinant of the submatrix corresponding to $k_1(g)>k_2(g)\hdots>k_t(g)$ in $H_d(b)$ will be non-zero which will contradict the maximality of $m_1>\hdots>m_t$. We will use this fact to prove the second claim.

To prove the second claim it suffices to prove that $K$ is invariant under the action of diagonal matrices and elementary upper triangular matrices. Recall, elementary upper triangular matrices are invertible upper triangular matrices with each diagonal entry being non-zero and only one other non-zero entry. The statement is trivial for diagonal matrices. Let $g\in \B(n,\F)$ be an elementary upper triangular matrix, such that $\leftidx{^g}{\!K_d}{}\not\subseteq K_d$. The basis of $K_d$ is the set of monomials $m_1>\hdots>m_t$. Recall, the grlex order satisfies $x_1>x_2\hdots>x_n$. As $g$ is upper triangular and the basis of $K_d$ is monomial, we see that in $H_d(g)$ all the non-zero entries in the column $m_i$ will all be in rows indexed by monomials at least as large as $m_i$ in the grlex order. This is the case because $\leftidx{^g}{\!m_i}{}$ will all be spanned by monomials larger or equal to $m_i$. As $g$ is elementary upper triangular for each column $m_i$, the row $m_i$ will have a non-zero entry. The previous two statements imply the set $M'(H_d(g))$ will contain the $t\times t$ submatrix corresponding to the rows $m_1>\hdots>m_t$ and the previous paragraph implies it will be the largest such tuple in this set.

If we perform column operations over the field $\F$ on $H_d(g)$ to produce the matrix $H'_d(g)$, the column span of $H'_d(g)$ is the same as the column span of $H_d(g)$ which is precisely $\leftidx{^g}{\!K_d}{}$. The set $M'(H'_d(g))$ will also be the same as $M'(H_d(g))$. We perform Gaussian elimination on $H_d(g)$ over the field $\F$ to produce a lower triangular matrix $H'_d(g)$. As this process may involve some column exchange operations we number the columns as $1,\hdots,t$. Let $k_i(g)$ be the largest monomial such that the element at column $i$ and row $k_i(g)$ be non-zero. By construction, $k_1(g)>\hdots>k_t(g)$. The following diagram might be helpful,

\begin{align*}
    H'_d(g)= \begin{blockarray}{cccc}
     1 & 2 & \hdots & t \\
     \begin{block}{(cccc)}
    \vdots & \vdots & \vdots & \vdots\\
    h'_{k_1(g),1} & 0 & \hdots & 0\\
    \vdots & \vdots & \vdots & \vdots\\
    h'_{k_2(g),1}   & h'_{k_2(g),2} & \hdots & 0\\
    \vdots & \vdots & \vdots & \vdots\\
    h'_{k_t(g),1}   & h'_{k_t(g),2} & \hdots & h'_{k_t(g),t}\\
    \vdots & \vdots & \vdots & \vdots\\
    \end{block}
    \end{blockarray}
    \qquad\text{\parbox{5cm}{Rows are indexed by monomials of degree $d$ in decreasing grlex order}}
\end{align*}

where $h'_{m,i}\in \F$ is the entry of the matrix $H'_d(g)$ in row $m$ and column $i$.

The sub-matrix corresponding to the rows $k_1(g),\hdots,k_t(g)$ will have a non-zero determinant and hence belong to $M'(H'_d(g))=M'(H_d(g))$. The previous paragraph implies $k_1(g)>\hdots>k_t(g)$ is smaller than $m_1>\hdots>m_t$. But as $H'_d(g)$ is lower triangular if $k_1(g)>\hdots>k_t(g)$ does not equal $m_1>\hdots>m_t$, the sub-matrix corresponding to the rows $m_1>\hdots>m_t$ will have zero determinant contradicting the fact that it lies in $M'(H_d(g))$. This implies $\leftidx{^g}{\!K_d}{}\subseteq K_d$. As $\leftidx{^g}{\!K_d}{}$ has the same dimension as $K_d$, they must be equal. This proves the second claim.

\end{proof}

\end{document}